\documentclass[11pt]{amsart}

\usepackage{amssymb}
\usepackage{amsmath}
\usepackage{amsfonts}
\usepackage{graphicx}
\usepackage{amsthm}
\usepackage{enumerate}
\usepackage[mathscr]{eucal}
\usepackage{mathrsfs}
\usepackage{verbatim}
\usepackage{yhmath}

\usepackage{color}

\makeatletter
\@namedef{subjclassname@2020}{%
  \textup{2020} Mathematics Subject Classification}
\makeatother

\numberwithin{equation}{section}
\numberwithin{figure}{section}

\theoremstyle{plain}
\newtheorem{theorem}{Theorem}[section]
\newtheorem{lemma}[theorem]{Lemma}
\newtheorem{proposition}[theorem]{Proposition}

\theoremstyle{plain}

\theoremstyle{remark}
\newtheorem{remark}[theorem]{Remark}

\DeclareMathOperator{\supp}{supp}

\DeclareMathOperator{\sgn}{sgn}

\allowdisplaybreaks[4]

\begin{document}

\title[A polynomial Roth theorem for corners]{A polynomial Roth theorem for corners in $\mathbb{R}^2$ and a related bilinear singular integral operator}

\author{Xuezhi Chen}
\address{Institute of Applied Physics \& Computational Mathematics\\
Beijing, 100088\\ P.R. China}
\email{xuezhi-chen@foxmail.com}

\author{Jingwei Guo}
\address{School of Mathematical Sciences\\
University of Science and Technology of China\\
Hefei, 230026\\ P.R. China}
\email{jwguo@ustc.edu.cn}

\date{}

\thanks{Jingwei Guo is supported by the NSF of Anhui Province, China (No. 2108085MA12). }



\begin{abstract}
We prove a quantitative Roth-type theorem for polynomial corners in $\mathbb{R}^2$. Let $P_1$ and $P_2$ be two linearly independent polynomials with zero constant term. We show that any measurable  subset of $[0,1]^2$ with positive measure contains three points $(x,y)$, $(x+P_1(t),y)$, $(x,y+P_2(t))$ with a gap estimate on $t$. We also prove boundedness results for a variant of the triangular Hilbert transform involving two polynomials and its associated maximal function. These results extend some earlier work of Christ, Durcik and Roos.

The key of the proof is to establish certain smoothing inequalities involving two polynomials. To accomplish that we give sublevel set estimates with general polynomials, explicit exponents and simplified proofs.
\end{abstract}

\maketitle

\section{Introduction}\label{sec1}

Christ \cite{C20} studied the trilinear form
\begin{equation*}
\mathcal{T}(f_1, f_2, f_3)=\int\limits_{\mathcal{M}} \prod_{j=1}^{3} f_j\left(x_j\right) \,\textrm{d}\mu (x),
\end{equation*}
where $x=(x_1,x_2,x_3)\in (\mathbb{R}^d)^3$, $\mathcal{M}$ is a submanifold of $(\mathbb{R}^d)^3$ of dimension $<3d$ and $\mu$ is a compactly supported measure on $\mathcal{M}$ with smooth density, and investigated smoothing inequalities of the form
\begin{equation}
\left|\mathcal{T}(f_1, f_2, f_3) \right|\leq C\prod_j \|f_j\|_{W^{p,s}}\quad \textrm{for some $s<0$}, \label{s1-8}
\end{equation}
where $W^{p,s}$ is the Sobolev space of functions having $s$ derivatives in $L^p$. See Christ and Zhou \cite{CZ22} for an application of trilinear smoothing inequalities in the study of certain maximal bilinear operators.

Recently Christ, Durcik and Roos \cite{CDR21} began the study of more singular situations. They studied the trilinear form
\begin{equation*}
\mathcal{T}(f_1, f_2, f_3)=\int\limits_{\mathbb{R}^3} f_1(x+t, y)f_2(x, y+t^2)f_3(x,y)\zeta(x,y,t) \,\textrm{d}x\textrm{d}y\textrm{d}t
\end{equation*}
with $d=2$, $\dim(\mathcal{M})=3$ and a $C_0^{\infty}$ function $\zeta$, and established a trilinear smoothing inequality of the type \eqref{s1-8}.
Among many interesting applications of this inequality,  they obtained a quantitative result on the existence of Roth-type patterns
\begin{equation*}
(x,y), (x+t, y), (x, y+t^2)
\end{equation*}
in subsets of $[0,1]^2$ of positive measure and also boundedness of the bilinear singular integral operator\footnote{It was pointed out in the Remark on \cite[P. 4]{CDR21} that their methods apply to the operator \eqref{s1-9} with $t$ and $t^2$ replaced by $[t]^\alpha$ and $[t]^\beta$ respectively, where $\alpha$, $\beta\geq 1$, $\alpha\neq \beta$ are real numbers and $[t]^\alpha$ represents $|t|^\alpha$ or $\sgn(t)|t|^\alpha$.}
\begin{equation}
\left(f_1,f_2\right)\mapsto\mathrm{p.v.}\int\limits_\mathbb{R}
f_1\left(x+t,y\right)f_2\left(x,y+t^2\right)\frac{\mathrm{d}t}{t} \label{s1-9}
\end{equation}
and associated maximal functions. See also Christ, Durcik, Kova\v{c} and Roos \cite{CDKR22} for an ergodic application. For quadrilinear smoothing inequalities and three-term sublevel set estimates see recent works by Christ \cite{C22a,C22b}.

Note that the operator \eqref{s1-9} is a variant of the well-known triangular Hilbert transform. While it is a hard open problem to determine boundedness properties of the triangular Hilbert transform, one can obtain certain boundedness of the  operator \eqref{s1-9} thanks to the existence of curvature (in the $t^2$ term).

In this paper we would like to further study the theory in \cite{CDR21} by generalizing $t$ and $t^2$ to two (linearly independent) polynomials. Throughout this paper we let $P_1$, $P_2:\mathbb{R}\rightarrow\mathbb{R}$ be polynomials with zero constant term, denoted by
\begin{equation}
	P_1(t)=a_{\sigma_1}t^{\sigma_1}+a_{\sigma_1+1}t^{\sigma_1+1}+\cdots+a_{d_1}t^{d_1}\label{s1-1}
\end{equation}
and
\begin{equation}	
	P_2(t)=b_{\sigma_2}t^{\sigma_2}+b_{\sigma_2+1}t^{\sigma_2+1}+\cdots+b_{d_2}t^{d_2},\label{s1-2}
\end{equation}
where $a_{\sigma_1}$, $a_{d_1}$, $b_{\sigma_2}$, $b_{d_2}$ are nonzero, $1\leq \sigma_1\leq d_1$ and $1\leq \sigma_2\leq d_2$.

Our first result is the following  triangular Roth-type theorem with two polynomials in the Euclidean setting.
\begin{theorem}\label{thm2}
	Let $P_1$, $P_2:\mathbb{R}\rightarrow\mathbb{R}$ be two linearly independent polynomials with zero constant term. Then for any $\varepsilon\in(0, 1/2)$, there exists a constant $c>0$ only depending on $P_1$, $P_2$, such that, given any measurable set $S\subset [0,1]^2$ with measure $|S|\geq \varepsilon$, it contains a triplet
	\begin{equation*}
		(x,y), (x+P_1(t),y), (x,y+P_2(t))
	\end{equation*}
	with
	\begin{equation*}
		t\geq \exp\left(-\exp\left(c\varepsilon^{-6}\right)\right). 
	\end{equation*}
\end{theorem}

The existence of such triples with worse $\varepsilon$-dependence follows from the multidimensional polynomial Szemer\'{e}di theorem of Bergelson and Leibman \cite{BL96}. Han, Lacey and Yang \cite{HLY21} studied such triples in the finite field setting. Besides, Shkredov \cite{S06a, S06b} studied linear triples $(x,y)$, $(x+t,y)$, $(x,y+t)$ in $\mathbb{N}^2$ and Durcik, Kova\v{c} and Rimani\'{c} \cite{DKR} studied such linear patterns in $(\mathbb{R}^d)^2$. For nonlinear Roth-type theorems in $\mathbb{R}$, see Bourgain \cite{Bourgain88}, Durcik, Guo and Roos \cite{DGR19}, the authors and Li \cite{CGL21}, Krause, Mirek, Peluse and Wright \cite{KMPW}, etc.


We also consider a triangular Hilbert transform involving two polynomials
\begin{equation*}
	T(f_1,f_2)(x,y)=\mathrm{p.v.}\int_\mathbb{R}\!f_1(x+P_1(t),y)f_2(x,y+P_2(t))\frac{\mathrm{d}t}{t}        
\end{equation*}
defined \textit{a priori} for test functions $f_1,f_2:\mathbb{R}^2\rightarrow \mathbb{C}$, and an associated bilinear maximal operator
\begin{equation*}
	M(f_1,f_2)(x,y)=\sup_{r>0}\frac{1}{2r}\int_{-r}^r\! \left|f_1(x+P_1(t),y)f_2(x,y+P_2(t))\right|\mathrm{d}t.
\end{equation*}

Our second result is on the boundedness of these operators.
\begin{theorem}\label{thm1}
	Let $P_1$, $P_2$ be two polynomials denoted by \eqref{s1-1} and \eqref{s1-2} with $\sigma_1\neq\sigma_2$ and $d_1\neq d_2$. Let $p,q\in(1,\infty)$ and $p^{-1}+q^{-1}=r^{-1}$. If $r\in[1,2)$, then $T$ extends to a bounded operator $L^p\times L^q\rightarrow L^r$. If $r\in[1,\infty)$, then $M$ extends to a bounded operator $L^p\times L^q\rightarrow L^r$.
\end{theorem}

\begin{remark}
If $r>1$ the boundedness of the maximal operator $M$ follows easily from H\"{o}lder's inequality and the boundedness of the Hardy-Littlewood maximal function.
\end{remark}

\begin{remark}\label{s1-3}
This kind of two-dimensional operators is connected with one-dimensional bilinear Hilbert transforms along curves. See \cite[\S 5.2]{CDR21} for how to derive one-dimensional results from this kind of two-dimensional ones.

The study of (one-dimensional) bilinear Hilbert transforms along curves was initiated by Li \cite{Li13}. Further investigations include, just to mention a few, Li and Xiao \cite{LX16} for polynomial curves and uniform $L^r$ bounds and Dong \cite{Dong19} for the variant involving two polynomials. The method in \cite{Li13} has later found interesting applications in the study of nonlinear Roth-type theorems in $\mathbb{R}$. See aforementioned references \cite{DGR19} and \cite{CGL21}.
\end{remark}

\begin{remark}
The range $r>(d-1)/d$ (with $d$ the degree of a given polynomial) for the $L^r$ bound given in \cite{LX16} is in general sharp up to the endpoint. In fact \cite{LX16} constructed examples of polynomial curves for which $r\geq (d-1)/d$ is necessary. For two-dimensional results as in Theorem \ref{thm1}, it would be interesting to ask what the largest range of $r$ is, and also be interesting to search for uniform bounds which are independent of coefficients of polynomials.
\end{remark}


For any $\Gamma>0$ and arbitrarily fixed $l\in \mathbb{Z}$ with $|l|>\Gamma$, we denote
\begin{equation*}
	\widetilde{P}_j(t)=\widetilde{P}_{j,l}(t)=2^{\mathfrak{r}_jl}P_j\left(2^{-l}t\right) \quad \text{for $j=1,2,$}
\end{equation*}
where we define $\mathfrak{r}_j=\sigma_j$ if $l>\Gamma$ and $\mathfrak{r}_j=d_j$ if $l<-\Gamma$. Note that when $t\asymp 1$ and $\Gamma$ is large, $\widetilde{P}_1(t)$ and $\widetilde{P}_2(t)$ behave like monomials $a_{\mathfrak{r}_1}t^{\mathfrak{r}_1}$ and $b_{\mathfrak{r}_2}t^{\mathfrak{r}_2}$ respectively.

Let $\zeta$ be a smooth function with compact support in $\mathbb{R}^2\times[1/2,2]$. Consider a bilinear operator (associated with $\widetilde{P}_1$ and $\widetilde{P}_2$)
\begin{equation}
	\widetilde{T}_l(f_1,f_2)(x,y)=\int_\mathbb{R} \! f_1\left(x+\widetilde{P}_1(t),y\right)f_2 \left(x,y+\widetilde{P}_2(t)\right)\zeta(x,y,t)\,\mathrm{d}t.\label{s3-1}
\end{equation}
We have the following decay estimate which plays the key role in the proofs of Theorems \ref{thm2} and \ref{thm1}.
\begin{theorem}\label{thm3}
Let $P_1$, $P_2$ be two linearly independent polynomials with zero constant term denoted by \eqref{s1-1} and \eqref{s1-2} respectively. If $\Gamma$ is sufficiently large (depending only on $P_1,P_2$), then there exist constants $\mathfrak{b}\geq0$ and $\sigma>0$ such that for all $|l|>\Gamma$ and $\lambda>1$ we have
\begin{equation}
	\left\|\widetilde{T}_l(f_1,f_2)\right\|_1\lesssim 2^{\mathfrak{b}|l|}\lambda^{-\sigma}\|f_1\|_2\|f_2\|_2 \label{s2-10}
\end{equation}
for all functions $f_1$, $f_2$ on $\mathbb{R}^2$ so that $\widehat{f_j}(\xi_1,\xi_2)$ is supported where $|\xi_j|\asymp\lambda$ for at least one index $j=1,2$. Moreover, if we assume $\mathfrak{r}_1\neq \mathfrak{r}_2$ in addition, then $\mathfrak{b}=0$ and $\sigma$ is an absolute constant.
\end{theorem}

The inequality \eqref{s2-10} is equivalent to the following smoothing inequality
\begin{equation*}
	\left\|\widetilde{T}_l(f_1,f_2)\right\|_1\lesssim 2^{\mathfrak{b}|l|} \|f_1\|_{H^{(-\sigma,0)}}\|f_2\|_{H^{(0,-\sigma)}}
\end{equation*}
for some constants $\mathfrak{b}\geq0$ and $\sigma>0$, where
\begin{equation*}
\|f\|^2_{H^{(a,b)}}=\int_{\mathbb{R}^2}\! \left|\widehat{f}(\xi_1,\xi_2) \right|^2 \left(1+|\xi_1|^2\right)^{a/2} \left(1+|\xi_2|^2\right)^{b/2} \,\textrm{d}\xi_1\textrm{d}\xi_2.
\end{equation*}

\begin{remark}
In Section \ref{sec2}, we will prove this key smoothing inequality with general polynomials.  As did in  Christ, Durcik and Roos \cite{CDR21} and in particular by applying their structural decomposition, after overcoming some technical difficulties, we reduce the proof to estimating measure of certain sublevel sets involving two general polynomials. Sublevel sets considered in \cite{CDR21} are associated to ($t$, $t^2$) and their treatment in \cite[\S 3.5]{CDR21} seems hard to generalize to general polynomials. They apply delicate decompositions and coordinate transformations adapted to integral curves of certain vector fields (associated to simple functions $t$, $t^2$) and apply tools from algebraic geometry to analytic functions. Instead, we use more elementary methods  which are readily applicable to general polynomials. We apply a relatively simple decomposition motivated by the form of certain derivative and Jacobian determinant; use a quantitative inverse function theorem to reduce the estimate to a local one; then use a change of variables to reduce it to a new sublevel set estimate associated to a function whose first derivative, essentially a polynomial, has a good lower bound; and also apply a multidimensional van der Corput theorem from Carbery, Christ and Wright \cite{CCW99} to polynomials.  Thanks to this treatment, we are able to provide sublevel set estimates with explicit exponents. For details see Propositions \ref{l4-2} and \ref{l4-1}.

Proposition \ref{l4-2} is already sufficient for our applications. But we provide a second version in Proposition \ref{l4-1} with an improved (absolute) exponent under an extra assumption $\sigma_1\neq \sigma_2$ if $l>\Gamma$, which may be of independent interest. We achieve this improvement by taking advantage of a nice structure of a complicated expression consisting of polynomials. See Remark \ref{remark3.2}. That expression comes out naturally from implicit differentiation as a factor of certain derivative. Its polynomial coefficients can be factored by Lemma \ref{app-1} into nice forms. If the extra assumption holds,  its size has a quite clear lower bound (see Lemma \ref{s5-5}) leading to a good estimate of certain sublevel set. If not, we are not  able to determine its lower bound. In that case we only have the result in Proposition \ref{l4-2}.
\end{remark}

\begin{remark}
 In Section \ref{sec4}, we prove Theorem \ref{thm1} by using the smoothing inequality in Theorem \ref{thm3} and bounds for an anisotropic operator in \cite[Theorem 2]{CDR21} (whose proof relies on ideas from Kova\v{c}'s work \cite{K12} on twisted paraproduct, Bernicot \cite{B12} and Durcik \cite{D15}). Particularly the bilinear estimate \eqref{s2-10} is used to deal with the high frequency part. Since general polynomials under consideration are lack of homogeneity, we have to give a generalization of \cite[Lemma 2.1]{CDR21}.  See Lemma \ref{lemma4.1}.

 In Section \ref{sec5}, we prove our Roth-type theorem, i.e. Theorem \ref{thm2}, by using Theorem \ref{thm3} and an argument of Bourgain from \cite{Bourgain88}.
\end{remark}

{\it Notations.} For real $X$ and nonnegative $Y$, we use $X\lesssim Y$ to denote $|X|\leq CY$ for some constant $C$. We write $X\lesssim_p Y$ to indicate that the implicit constant $C$ depends on a parameter $p$. If $X$ is nonnegative, $X\gtrsim Y$ means $Y\lesssim X$. The Landau notation $X=O_p(Y)$ is equivalent to $X\lesssim_p Y$. The notation $X\asymp Y$ means that $X\lesssim Y$ and $Y\lesssim X$. The Fourier transforms in the Euclidean space are
\begin{equation*}
\widehat{f}(\xi)=\int_{\mathbb{R}^d} \! f(x)\exp\left(-2\pi i \xi\cdot x\right) \,\textrm{d}x, \quad g^\vee(x)=\int_{\mathbb{R}^d} \! g(\xi)\exp\left(2\pi i x\cdot \xi\right) \,\textrm{d}\xi.
\end{equation*}
For a function $f$ on $\mathbb{R}^2$ and a function $\phi$ on $\mathbb{R}$, partial convolutions are given by $\phi\!*_1\! f(x,y)=\int_{\mathbb{R}}\! f(x-u,y)\phi(u)\,\textrm{d}u$ and $\phi\!*_2\! f(x,y)=\int_{\mathbb{R}}\! f(x,y-u)\phi(u)\,\textrm{d}u$.
We set $\mathbb{N}_0=\mathbb{N}\cup \{0\}$ and let $\textbf{1}_{E}$ represent the characteristic function of a set $E$. The notation $a \gg (\ll)$ $b$ means that $a$ is much greater (less) than $b$.



\section{Estimate of the bilinear operator} \label{sec2}

In this section we mainly follow the strategy used in Christ, Durcik and Roos' \cite[Section 3]{CDR21} to prove Theorem \ref{thm3}. We have to overcome difficulties caused by the generalization from $t, t^2$ to two polynomials $\widetilde{P}_1, \widetilde{P}_2$. In particular, a key new ingredient here is the application of certain sublevel set estimates involving two polynomials. We will provide such estimates with explicit exponents and simplified proofs in the next section.

We first observe that there exists a constant $\mathfrak{b}_1\geq0$ (which is $0$ if $\mathfrak{r}_1\neq \mathfrak{r}_2$ and positive if $\mathfrak{r}_1=\mathfrak{r}_2$) such that, whenever $|l|>\Gamma$ is sufficiently large,
\begin{equation}
	\left\|\widetilde{T}_l(f_1,f_2)\right\|_1\lesssim 2^{\mathfrak{b}_1|l|}\|f_1\|_{3/2}\|f_2\|_{3/2} \label{s1-10}
\end{equation}
for all functions $f_1$, $f_2$. Indeed,
\begin{align*}
\left\|\widetilde{T}_l(f_1,f_2)\right\|_1
&\leq\int \!|f_1(x,y)|\left(\int \! \left|f_2\left(x-\widetilde{P}_1(t),y+\widetilde{P}_2(t)\right)\right|\tau(t)\,\mathrm{d}t\!\right)\mathrm{d}x\mathrm{d}y\\
&\leq \|f_1\|_{3/2}\left\|\int \!
\left|f_2\left(x-\widetilde{P}_1(t),y+\widetilde{P}_2(t)\right)\right|\tau(t)\,\mathrm{d}t \right\|_{L^3(\mathrm{d}x\mathrm{d}y)},
\end{align*}
where $\tau(t)=\tau(x,y,t)=|\zeta(x-\widetilde{P}_1(t),y,t)|$, as a function of $t$, is nonnegative, smooth and  compactly supported on $[1/2,2]$. We then apply the following result to estimate the above $L^3$-norm.
\begin{lemma}[{Oberlin, \cite[Theorem 1]{Oberlin02}}]
	Fix a positive integer $D$. There is a positive constant $C(D)$ such that if $p_1(t)$ and $p_2(t)$ are two real-valued polynomials of degree not exceeding $D$ and $\mu$ is the measure on the curve $(p_1(t),p_2(t))$, $-\infty<t<\infty$, given by
	\begin{equation*}
		\left|p_1'(t)p_2''(t)-p_1''(t)p_2'(t)\right|^{1/3}\,\mathrm{d}t,
	\end{equation*}
then
\begin{equation*}
	\|\mu*f\|_{L^3(\mathbb{R}^2)}\leq C(D)\|f\|_{L^{3/2}(\mathbb{R}^2)}  
\end{equation*}
for all functions $f$ on $\mathbb{R}^2$.
\end{lemma}
\noindent Plugging in $f=|f_2|$, $p_1(t)=\widetilde{P}_1(t)$ and $p_2(t)=-\widetilde{P}_2(t)$ and observing that, whenever $|l|>\Gamma$ is sufficiently large,
\begin{equation*}
	\tau(t)\lesssim 2^{\mathfrak{b}_1|l|}\left|p_1'(t)p_2''(t)-p_1''(t)p_2'(t)\right|^{1/3}, \quad t\in [1/2,2],
\end{equation*}
with a constant $\mathfrak{b}_1$ being $0$ if $\mathfrak{r}_1\neq \mathfrak{r}_2$ and  positive if $\mathfrak{r}_1= \mathfrak{r}_2$, we then obtain
\begin{equation*}
\left\|\int \!
\left|f_2\left(x-\widetilde{P}_1(t),y+\widetilde{P}_2(t)\right)\right|\tau(t)\,\mathrm{d}t \right\|_{L^3(\mathrm{d}x\mathrm{d}y)}\lesssim 2^{\mathfrak{b}_1|l|}\|f_2\|_{3/2},
\end{equation*}
which immediately leads to \eqref{s1-10}.

In view of \eqref{s1-10}, by an interpolation argument, it then suffices to prove that there exist constants $\mathfrak{b}_2\geq0$ (which is $0$ if $\mathfrak{r}_1\neq \mathfrak{r}_2$ and positive if $\mathfrak{r}_1=\mathfrak{r}_2$) and $\sigma>0$ such that, whenever $|l|>\Gamma$ is sufficiently large and $\lambda>1$,
\begin{equation}
	\left\|\widetilde{T}_l(f_1,f_2)\right\|_1\lesssim 2^{\mathfrak{b}_2|l|} \lambda^{-\sigma}\|f_1\|_\infty\|f_2\|_\infty   \label{s3-3}
\end{equation}
holds for all functions $f_1$, $f_2$ so that $\widehat{f_j}(\xi_1,\xi_2)$ is supported where $|\xi_j|\asymp\lambda$ for at least one index $j=1,2$.

Without loss of generality, we only prove \eqref{s3-3} for sufficiently large $\lambda$.   Assume that
\begin{equation}\label{s3-4}
	\widehat{f_{j_*}}(\xi_1,\xi_2)\neq 0\quad \Longrightarrow\quad \lambda\leq|\xi_{j_*}|\leq 2\lambda
\end{equation}
either holds for $j_*=1$ or $j_*=2$. By a Littlewood-Paley decomposition, we may also assume that for $j=1,2$
\begin{equation}\label{s3-5}
	\widehat{f_j}(\xi_1,\xi_2)\neq 0\quad \Longrightarrow\quad |\xi_j|\leq 2\lambda.
\end{equation}
As a consequence, the Bernstein's inequality implies that
\begin{equation}
	\left\|\partial_{x_j}^k f_j\right\|_\infty\lesssim \lambda^k\|f_j\|_\infty ,\quad  \textrm{for $j=1,2$}.\label{s3-6}
\end{equation}


\subsection{A basic estimate}\label{sec3-1}
Let $\eta$ be a smooth nonnegative bump function compactly supported in a small neighborhood of $[-1/2,1/2]^2$ satisfying
\begin{equation*}
	\sum_{m\in\mathbb{Z}^2}\eta((x,y)-m)=1, \quad \textrm{for all $(x,y)\in\mathbb{R}^2$}.
\end{equation*}
Let $\gamma\in(1/2,1)$ be a parameter to be determined later. Set $\eta_m(x,y)=\eta(\lambda^\gamma(x,y)-m)$. By using the partition of unity  $\sum_{m\in\mathbb{Z}^2}\eta_m=1$, we have the decomposition
\begin{equation*}
	f_j=\sum_{m\in\mathbb{Z}^2}f_{j,m}\quad \text{with} \quad f_{j,m}=\eta_mf_j.
\end{equation*}
We observe that the function $f_{j,m}$ is supported in a small neighborhood of the cube of side length $\lambda^{-\gamma}$ centered at $\lambda^{-\gamma}m$. Denote by $Q_m$ the cube of side length $2\lambda^{-\gamma}$ centered at $\lambda^{-\gamma}m$.

By using the decomposition of $f_j$, we write
 \begin{equation}\label{s3-7}
 	\widetilde{T}_l(f_1,f_2)=\sum_{\mathbf{m}\in(\mathbb{Z}^2)^2}\widetilde{T}_l\left(f_{1,m_1},f_{2,m_2}\right) \quad \textrm{with $\mathbf{m}=(m_1,m_2)$}.
 \end{equation}
Let
\begin{equation*}
\mathcal{M}=\left\{(m_1,m_2)\in(\mathbb{Z}^2)^2 : 	\left\|\widetilde{T}_l(f_{1,m_1},f_{2,m_2})\right\|_1\neq 0\right\}
\end{equation*}
and, for $j=1,2$,
\begin{equation*}
\mathcal{M}_j=\{m_j\in\mathbb{Z}^2 : \textrm{there exists $(m_1,m_2)\in\mathcal{M}$}\}.
\end{equation*}
Arguing as in \cite{CDR21} yields that
\begin{equation*}
\#\mathcal{M}_j\lesssim\lambda^{2\gamma} \textrm{ and }	\#\mathcal{M}\lesssim \lambda^{3\gamma},
\end{equation*}
where both implicit constants do not depend on $l$. By the triangle inequality, we have
\begin{equation*}
	\left\|\widetilde{T}_l(f_1,f_2)\right\|_1\leq\sum_{\mathbf{m}\in\mathcal{M}}\left\|\widetilde{T}_l(f_{1,m_1},f_{2,m_2})\right\|_1.
\end{equation*}

For each $\mathbf{m}=(m_1,m_2)\in\mathcal{M}$, fix a point $(\bar{x},\bar{y},\bar{t})=(\bar{x}_\mathbf{m},\bar{y}_\mathbf{m},\bar{t}_\mathbf{m})$ such that
\begin{equation}\label{s2-11}
	\left(\bar{x}+\widetilde{P}_1(\bar{t}),\bar{y}\right)\in Q_{m_1} \textrm{ and } \left(\bar{x},\bar{y}+\widetilde{P}_2(\bar{t})\right)\in Q_{m_2}.
\end{equation}
Then for each $(x,y,t)$ with $(x+\widetilde{P}_1(t),y)\in Q_{m_1}$ and $(x,y+\widetilde{P}_2(t))\in Q_{m_2}$, we have
\begin{equation*}
	|x-\bar{x}|+|y-\bar{y}|+|t-\bar{t}|\lesssim\lambda^{-\gamma}.
\end{equation*}
Thus $\widetilde{T}_l(f_{1,m_1},f_{2,m_2})(x,y)$ can be written as
\begin{equation*}
\int_\mathbb{R}\!f_{1,m_1}\left(x+\widetilde{P}_1(t),y\right)f_{2,m_2}\left(x,y+\widetilde{P}_2(t)\right)\zeta_\mathbf{m}(x,y,t)\,\mathrm{d}t,
\end{equation*}
where $\zeta_\mathbf{m}$ is a smooth function compactly supported in the intersection of $\mathbb{R}^2\times[1/2,2]$ and a cube with side lengths $O(\lambda^{-\gamma})$, which satisfies
 \begin{equation*}
 	\left\|\partial^{\alpha}\zeta_\mathbf{m}\right\|_\infty\lesssim \lambda^{\gamma|\alpha|}
 \end{equation*}
for each multi-index $\alpha\in \mathbb{N}_0^3$. By the Cauchy-Schwarz inequality, we then find that $\|\widetilde{T}_l(f_{1,m_1},f_{2,m_2})\|_1^2$ is majorized by
\begin{equation}
	\begin{split}
		\lambda^{-2\gamma}\!&\int_{\mathbb{R}^4}\!f_{1,m_1}\left(x+\widetilde{P}_1(t+s),y\right)\overline{f_{1,m_1}\left(x+\widetilde{P}_1(t),y\right)}\\
		&f_{2,m_2}\left(x,y+\widetilde{P}_2(t+s)\right)\overline{f_{2,m_2}\left(x,y+\widetilde{P}_2(t)\right)}\widetilde{\zeta}_\mathbf{m}(x,y,t,s)\,\mathrm{d}x\mathrm{d}y\mathrm{d}t\mathrm{d}s,
	\end{split}	\label{s3-10}
\end{equation}
where $\widetilde{\zeta}_\mathbf{m}(x,y,t,s)=\zeta_\mathbf{m}(x,y,t+s)\overline{\zeta_\mathbf{m}(x,y,t)}$. Then for each $(x,y,t,s)$ in the support of $\widetilde{\zeta}_\textbf{m}$, the quantities $|x-\bar{x}|$, $|y-\bar{y}|$, $|t-\bar{t}|$ and $|s|$ are $O(\lambda^{-\gamma})$. As a consequence, we can get by the mean value theorem that
\begin{equation*}
	\widetilde{P}_j(t+s)=\widetilde{P}_j(t)+\widetilde{P}_j'(\bar{t})s+O(\lambda^{-2\gamma}), \quad  j=1,2,
\end{equation*}
and by \eqref{s3-6} that
\begin{equation*}
	f_{1,m_1}\left(x+\widetilde{P}_1(t+s),y\right)=f_{1,m_1}\left(x+\widetilde{P}_1(t)+\widetilde{P}_1'(\bar{t})s,y\right)+O\left(\lambda^{1-2\gamma}\|f_1\|_\infty\right)
\end{equation*}
and
\begin{equation*}
	f_{2,m_2}\left(x,y+\widetilde{P}_2(t+s)\right)=f_{2,m_2}\left(x,y+\widetilde{P}_2(t)+\widetilde{P}_2'(\bar{t})s\right)+O\left(\lambda^{1-2\gamma}\|f_2\|_\infty\right).
\end{equation*}
Denote that
\begin{equation*}
	\mathcal{D}_s^{(1)}f(x,y)=f(x+s,y)\overline{f(x,y)}\,\text{ and }\,\mathcal{D}_s^{(2)}f(x,y)=f(x,y+s)\overline{f(x,y)}.
\end{equation*}
Thus
\begin{equation*}
	\eqref{s3-10}\lesssim \lambda^{-2\gamma}\left|\int_{|s|\lesssim\lambda^{-\gamma}}\!B_\mathbf{m}(s)\,\mathrm{d}s\right|+\lambda^{1-8\gamma}\|f_1\|_\infty^2\|f_2\|_\infty^2,
\end{equation*}
where
\begin{equation}
	\begin{split}
		B_\mathbf{m}(s)=\int_{\mathbb{R}^3}\!&\mathcal{D}_{\widetilde{P}_1'(\bar{t})s}^{(1)}f_{1,m_1}\left(x+\widetilde{P}_1(t),y\right)\mathcal{D}_{\widetilde{P}_2'(\bar{t})s}^{(2)}f_{2,m_2}\left(x,y+\widetilde{P}_2(t)\right)\\
		&\widetilde{\zeta}_\mathbf{m}(x,y,t,s)\,\mathrm{d}x\mathrm{d}y\mathrm{d}t.
	\end{split}\label{s3-13}
\end{equation}
This leads to
\begin{equation}\label{s3-11}
	\left\|\widetilde{T}_l(f_1,f_2)\right\|_1\lesssim  \lambda^{-\gamma} \sum_{\mathbf{m}\in\mathcal{M}}\left|\int_{|s|\lesssim\lambda^{-\gamma}}\!\!B_\mathbf{m}(s)\,\mathrm{d}s\right|^{1/2}+\lambda^{\frac{1}{2}-\gamma}\|f_1\|_\infty\|f_2\|_\infty.
\end{equation}
Applying the Cauchy-Schwarz inequality yields that the first term on the right of \eqref{s3-11} is majorized by
\begin{equation}\label{s3-12}
 \lambda^{\frac{\gamma}{2}}\left(\sum_{\mathbf{m}\in\mathcal{M}}\int_{|s|\lesssim\lambda^{-\gamma}}\!\!\!\!|B_\mathbf{m}(s)|\,\mathrm{d}s\right)^{1/2}.
\end{equation}

We next introduce the Fourier series into the expression of $B_\mathbf{m}(s)$. Since for every $s\in\mathbb{R}$ and $j=1,2$, the function $\mathcal{D}_s^{(j)}f_{j,m_j}$ is supported in a cube of side length $1.5\lambda^{-\gamma}$ concentrically contained in $Q_{m_j}$, we can express it as
\begin{equation}\label{s3-14}
	\mathcal{D}_s^{(j)}f_{j,m_j}(x,y)=\sum_{k\in\mathbb{Z}^2}a_{j,m_j,k,s}e^{\pi i\lambda^\gamma k\cdot(x,y)} \textrm{ for $(x,y)\in Q_{m_j}$},
\end{equation}
where the Fourier coefficients are given by
\begin{equation*}
	a_{j,m_j,k,s}=\frac{1}{4}\lambda^{2\gamma}\widehat{\mathcal{D}_s^{(j)}f_{j,m_j}}\left(\frac{1}{2}\lambda^{\gamma}k\right),
\end{equation*}
satisfying, by the Parseval's identity, that
\begin{equation*}
	\sum_{k\in\mathbb{Z}^2}\left|a_{j,m_j,k,s}\right|^2=\frac{1}{4}\lambda^{2\gamma}\left\|\mathcal{D}_s^{(j)}f_{j,m_j}\right\|_2^2.
\end{equation*}
Moreover, as a consequence of the band-limitedness hypotheses on $f_j$, we can get the following bounds. Let $k=(k_1,k_2)\in\mathbb{Z}^2$. Integration by parts gives
\begin{equation*}
	\sum_{k_1\in\mathbb{Z}}\left|a_{2,m_2,k,s}\right|^2\lesssim_N \lambda^{(1-\gamma)N}|k_2|^{-N}\|f_2\|_\infty^4,
\end{equation*}
for every $N\geq 1$ and $k_2\neq 0$. Let $0<\epsilon_1<\epsilon_2$ be two small parameters to be determined later. It follows from the inequality above that
\begin{equation}\label{s4-21}
	\sum_{|k_2|\geq\lambda^{1-\gamma+\epsilon_2}}\sum_{k_1\in\mathbb{Z}}\left|a_{2,m_2,k,s}\right|^2\lesssim_{N,\epsilon_2}\lambda^{-N}\|f_2\|_\infty^4.
\end{equation}
Similarly we have
\begin{equation}\label{s4-22}
	\sum_{|k_1|\geq\lambda^{1-\gamma+\epsilon_1}}\sum_{k_2\in\mathbb{Z}}\left|a_{1,m_1,k,s}\right|^2\lesssim_{N,\epsilon_1}\lambda^{-N}\|f_1\|_\infty^4.
\end{equation}

In the rest part, we denote $k_j=(k_{j,1},k_{j,2})\in\mathbb{Z}^2$, $j=1,2$. We decompose $(\mathbb{Z}^2)^2$ into several parts
\begin{align*}
	L_1&=\left\{(k_1,k_2)\in(\mathbb{Z}^2)^2:|k_{1,1}|\geq\lambda^{1-\gamma+\epsilon_1}\right\},\\
	L_2&=\left\{(k_1,k_2)\in(\mathbb{Z}^2)^2:|k_{1,1}|<\lambda^{1-\gamma+\epsilon_1}\textrm{ and  }|k_{2,2}|\geq\lambda^{1-\gamma+\epsilon_2}\right\},\\
	L_3&=\left\{(k_1,k_2)\in(\mathbb{Z}^2)^2:|k_{1,1}|<\lambda^{1-\gamma+\epsilon_1}\textrm{ and  }|k_{2,2}|<\lambda^{1-\gamma+\epsilon_2}\right\}.
\end{align*}
Moreover, let $\epsilon_3>0$ be a small parameter (to be determined later) and we further decompose $L_3$ into $L_{3,1}$ and $L_{3,2}$ defined by
\begin{equation*}
	L_{3,1}=\left\{(k_1,k_2)\in L_3:|k_1+k_2|\lesssim\lambda^{\epsilon_3}\textrm{ and }|\widetilde{P}_1'(\bar{t})k_{1,1}+\widetilde{P}_2'(\bar{t})k_{2,2}|\lesssim\lambda^{\epsilon_3}\right\}
\end{equation*}
and
\begin{equation*}
		L_{3,2}=L_3\setminus L_{3,1}.
\end{equation*}
By using \eqref{s3-14} and the above index sets, we rewrite and decompose $B_\mathbf{m}(s)$ as follows\footnote{We may assume $\widetilde{\zeta}_\mathbf{m}$ still contains the information of supports of the functions $\mathcal{D}_{\widetilde{P}_1'(\bar{t})s}^{(1)}f_{1,m_1}(x+\widetilde{P}_1(t),y)$ and $\mathcal{D}_{\widetilde{P}_2'(\bar{t})s}^{(2)}f_{2,m_2}(x,y+\widetilde{P}_2(t))$.}
\begin{align}
	B_\mathbf{m}(s)&=\sum_{(k_1,k_2)\in (\mathbb{Z}^2)^2}a_{1,m_1,k_1,\widetilde{P}_1'(\bar{t})s}a_{2,m_2,k_2,\widetilde{P}_2'(\bar{t})s}\cdot \nonumber\\
	&\qquad\qquad \int_{\mathbb{R}^3}\!e^{\pi i \lambda^\gamma [k_1\cdot (x+\widetilde{P}_1(t),y)+k_2\cdot(x,y+\widetilde{P}_2(t))]}\widetilde{\zeta}_\mathbf{m}(x,y,t,s)\,\mathrm{d}x\mathrm{d}y\mathrm{d}t\label{s2-7}\\
	&=\sum_{(k_1,k_2)\in L_1}+\sum_{(k_1,k_2)\in L_2}+\sum_{(k_1,k_2)\in L_{3,1}}+\sum_{(k_1,k_2)\in L_{3,2}}\nonumber\\
	&=:B_1+B_2+B_{3,1}+B_{3,2}.\nonumber
\end{align}

By using the Cauchy-Schwarz inequality, integration by parts, \eqref{s4-21} and \eqref{s4-22}, we can get that
\begin{equation}\label{s4-23}
	|B_1|+|B_2|+|B_{3,2}|\lesssim_{N,\epsilon_1,\epsilon_2,\epsilon_3}\lambda^{-N}\|f_1\|_\infty^2\|f_2\|_\infty^2.
\end{equation}

\begin{proof}[Proof of \eqref{s4-23}]
Firstly, we notice that
\begin{equation}
	\begin{split}
	B_1=\sum_{\stackrel{k_1\in\mathbb{Z}^2}{|k_{1,1}|\geq\lambda^{1-\gamma+\epsilon_1}}}
	&a_{1,m_1,k_1,\widetilde{P}_1'(\bar{t})s}\int_{\mathbb{R}^3}\!\mathcal{D}_{\widetilde{P}_2'(\bar{t})s}^{(2)}f_{2,m_2}\left(x,y+\widetilde{P}_2(t)\right)\\
	&e^{\pi i \lambda^\gamma k_1\cdot(x+\widetilde{P}_1(t),y)}\widetilde{\zeta}_\mathbf{m}(x,y,t,s)\,\mathrm{d}x\mathrm{d}y\mathrm{d}t.
    \end{split}\label{s2-8}
\end{equation}
Since $\widetilde{P}_1'(t)\asymp 1$ for $t\asymp 1$,
the $(y,t)$-gradient of the phase function $k_1\cdot(x+\widetilde{P}_1(t),y)$ is equal to $(k_{1,2}, \widetilde{P}_1'(t)k_{1,1})$, whose absolute value is $\asymp |k_1|$. Applying integration by parts in the $(y,t)$-direction and trivial estimate in the $x$-direction gives that the integral in \eqref{s2-8} is $\lesssim\lambda^{2-5\gamma}\|f_2\|_\infty^2|k_1|^{-2}$. Then by the Cauchy-Schwarz inequality and \eqref{s4-22}, we obtain
\begin{align*}
	|B_1|&\lesssim\lambda^{2-5\gamma}\|f_2\|_\infty^2\bigg(\sum_{\stackrel{k_{1,2}\in\mathbb{Z}}{|k_{1,1}|\geq\lambda^{1-\gamma+\epsilon_1}}}\left|a_{1,m_1,k_1,\widetilde{P}_1'(\bar{t})s}\right|^2\bigg)^{1/2}\\
	&\lesssim_{N,\epsilon_1}\lambda^{-N}\|f_1\|_\infty^2\|f_2\|_\infty^2.
\end{align*}

Secondly, to estimate $B_2$ and $B_{3,2}$ we use the formula in \eqref{s2-7}. The gradient of the phase function $k_1\cdot \left(x+\widetilde{P}_1(t),y\right)+k_2\cdot\left(x,y+\widetilde{P}_2(t)\right)$ is equal to
\begin{equation}\label{s3-20}
\left(k_1+k_2,\widetilde{P}_1'(t)k_{1,1}+\widetilde{P}_2'(t)k_{2,2}\right).
\end{equation}

For $B_2$, the absolute value of \eqref{s3-20} is
\begin{equation*}
\asymp \sqrt{|k_1+k_2|^2+|k_{2,2}|^2}
\end{equation*}
since $|\widetilde{P}_2'(t)k_{2,2}|\gtrsim\lambda^{1-\gamma+\epsilon_2}\gg \lambda^{1-\gamma+\epsilon_1}\gtrsim |\widetilde{P}_1'(t)k_{1,1}|$. Integration by parts gives that the integral in \eqref{s2-7} is
\begin{equation*}
	\lesssim\frac{\lambda ^{-3\gamma}}{|k_1+k_2|^2+|k_{2,2}|^2}.
\end{equation*}
By using the Cauchy-Schwarz inequality and \eqref{s4-21}, we then get
\begin{equation*}
	|B_2|\lesssim_{N,\epsilon_2}\lambda^{-N}\|f_1\|_\infty^2\|f_2\|_\infty^2.
\end{equation*}

For $B_{3,2}$, we have $(k_1,k_2)\in L_{3,2}$ and thus the absolute value of the gradient \eqref{s3-20} is
\begin{equation*}
	\gtrsim \max\left\{|k_1+k_2|, \left|\widetilde{P}_1'(\bar{t})k_{1,1}+\widetilde{P}_2'(\bar{t})k_{2,2}\right|\right\}.
\end{equation*}
By using integration by parts and the Cauchy-Schwarz inequality we get
\begin{align*}
	|B_{3,2}|\lesssim_N&\lambda^{-3\gamma}\bigg(\sum_{k_1,k_2\in \mathbb{Z}^2}\left|a_{1,m_1,k_1,\widetilde{P}_1'(\bar{t})s}\right|^2\left|a_{2,m_2,k_2,\widetilde{P}_2'(\bar{t})s}\right|^2\bigg)^{1/2}\cdot\\
	&\bigg(\sum_{(k_1,k_2)\in L_{3,2}}\!\!\!\max\left\{|k_1+k_2|, \left|\widetilde{P}_1'(\bar{t})k_{1,1}+\widetilde{P}_2'(\bar{t})k_{2,2}\right|\right\}^{-2N}\bigg)^{1/2}\\
	\lesssim_{N,\epsilon_3}&\lambda^{-N}\|f_1\|_\infty^2\|f_2\|_\infty^2.
\end{align*}
This finishes the proof of \eqref{s4-23}.
\end{proof}

It remains to estimate the size of $|B_{3,1}|$. We have that
\begin{align*}
	\left|B_{3,1}\right|\lesssim& \lambda^{-3\gamma}\sum_{(k_1,k_2)\in L_{3,1}}\left|a_{1,m_1,k_1,\widetilde{P}_1'(\bar{t})s}\right|\left|a_{2,m_2,k_2,\widetilde{P}_2'(\bar{t})s}\right|\\
	\leq&\lambda^{-3\gamma}\!\!\!\!\!\!\sum_{\stackrel{\Delta k\in\mathbb{Z}^2}{|\Delta k|\lesssim\lambda^{\epsilon_3}}}\,\,\,\sum_{\stackrel{k_1\in\mathbb{Z}^2}{|\widetilde{P}_1'(\bar{t})k_{1,1}-\widetilde{P}_2'(\bar{t})k_{1,2}|\lesssim\lambda^{\epsilon_3}}}\!\!\!\!\!\!\!\!\!\!\!\!\!\!\!\left|a_{1,m_1,k_1,\widetilde{P}_1'(\bar{t})s}\right|\left|a_{2,m_2,\Delta k-k_1,\widetilde{P}_2'(\bar{t})s}\right|\\
	\lesssim&\lambda^{-3\gamma+2\epsilon_3}\bigg(\sum_{\stackrel{k_1\in\mathbb{Z}^2}{|\widetilde{P}_1'(\bar{t})k_{1,1}-\widetilde{P}_2'(\bar{t})k_{1,2}|\lesssim\lambda^{\epsilon_3}}}\!\!\!\!\!\!\!\!\!\!\!\!\!\!\!\left|a_{1,m_1,k_1,\widetilde{P}_1'(\bar{t})s}\right|^2\bigg)^{1/2}F_{2,m_2}^{1/2}\left(\widetilde{P}_2'(\bar{t})s\right),
\end{align*}
where we denote
\begin{equation*}
F_{j,m_j}(s)=\frac{1}{4}\lambda^{2\gamma}\left\|\mathcal{D}_s^{(j)}f_{j,m_j}\right\|_2^2\quad \textrm{for $j=1,2$}.
\end{equation*}
Similarly, we have
\begin{equation*}
	|B_{3,1}|\lesssim\lambda^{-3\gamma+2\epsilon_3}\bigg(\sum_{\stackrel{k_2\in\mathbb{Z}^2}{|\widetilde{P}_1'(\bar{t})k_{2,1}-\widetilde{P}_2'(\bar{t})k_{2,2}|\lesssim\lambda^{\epsilon_3}}}\!\!\!\!\!\!\!\!\!\!\!\!\!\!\!\left|a_{2,m_2,k_2,\widetilde{P}_2'(\bar{t})s}\right|^2\bigg)^{1/2}F_{1,m_1}^{1/2}\left(\widetilde{P}_1'(\bar{t})s\right).
\end{equation*}
To sum up, we get for $j=1,2$ and $\underline{j}:=3-j$ that
\begin{equation}\label{s3-25}
	\begin{split}
		|B_\mathbf{m}(s)|\lesssim&\lambda^{-N}\|f_1\|_\infty^2\|f_2\|_\infty^2+\lambda^{-3\gamma+2\epsilon_3}\cdot\\
		&\bigg(\sum_{\stackrel{k_j\in\mathbb{Z}^2}{|\widetilde{P}_1'(\bar{t})k_{j,1}-\widetilde{P}_2'(\bar{t})k_{j,2}|\lesssim\lambda^{\epsilon_3}}}\!\!\!\!\!\!\!\!\!\!\!\left|a_{j,m_j,k_j,\widetilde{P}_j'(\bar{t})s}\right|^2\bigg)^{1/2}F_{\underline{j},m_{\underline{j}}}^{1/2}\left(\widetilde{P}_{\underline{j}}'(\bar{t})s\right).
	\end{split}
\end{equation}

By using \eqref{s3-11}, \eqref{s3-12} and \eqref{s3-25},  we get for $j=1,2$  that
\begin{equation*}
	\begin{split}
		\left\|\widetilde{T}_l(f_1,f_2)\right\|_1\lesssim&\lambda^{\frac{1}{2}-\gamma}\|f_1\|_\infty\|f_2\|_\infty+\lambda^{-\gamma+\epsilon_3}\cdot\\
		\bigg(\!\sum_{\mathbf{m}\in\mathcal{M}}&\!\int_{|s|\lesssim\lambda^{-\gamma}}\!\!\!\bigg(\!\!\!\!\!\!\!\!\!\!\!\!\!\!\!\sum_{\stackrel{k_j\in\mathbb{Z}^2}{|\widetilde{P}_1'(\bar{t})k_{j,1}-\widetilde{P}_2'(\bar{t})k_{j,2}|\lesssim\lambda^{\epsilon_3}}}\!\!\!\!\!\!\!\!\!\!\!\!\!\!\!\!\!\!\!\!\left|a_{j,m_j,k_j,\widetilde{P}_j'(\bar{t})s}\right|^2\bigg)^{\!1/2}\!\!F_{\underline{j},m_{\underline{j}}}^{1/2}\left(\widetilde{P}_{\underline{j}}'(\bar{t})s\right)\mathrm{d}s\bigg)^{\!1/2}\!.
	\end{split}
\end{equation*}
We then use the Cauchy-Schwarz inequality and change of variables to get
\begin{equation}\label{s3-26}
	\begin{split}
		\left\|\widetilde{T}_l(f_1,f_2)\right\|_1\lesssim&\lambda^{\frac{1}{2}-\gamma}\|f_1\|_\infty\|f_2\|_\infty+\lambda^{-\frac{1}{4}\gamma+\epsilon_3}G_{\underline{j}}^{1/4}\cdot\\
		&\bigg(\!\int_{|s|\lesssim\lambda^{-\gamma}}\!\!\sum_{\mathbf{m}\in\mathcal{M}}\!\!\!\!\!\sum_{\stackrel{k_j\in\mathbb{Z}^2}{|\widetilde{P}_1'(\bar{t})k_{j,1}-\widetilde{P}_2'(\bar{t})k_{j,2}|\lesssim\lambda    ^{\epsilon_3}}}\!\!\!\!\!\!\!\!\!\!\!\!\!\!\!\left|a_{j,m_j,k_j,s}\right|^2\,\mathrm{d}s\!\bigg)^{1/4},
	\end{split}
\end{equation}
where
\begin{equation*}
G_j=\sup_{m_j\in\mathcal{M}_j}\int_{|s|\lesssim\lambda^{-\gamma}}F_{j,m_j}(s)\,\mathrm{d}s.
\end{equation*}

Concerning the domain of summation over $\mathbf{m}$ and $k_j$ in \eqref{s3-26},  we claim that for  any fixed $j\in\{1,2\}$, $m_j$ and $k_j\in\mathbb{Z}^2$ with $k_{j,j}\neq0$, there are at most
\begin{equation*}
	1+O\left(2^{\mathfrak{c}|l|}\lambda^{\gamma+\epsilon_3}|k_{j,j}|^{-1}\right)
\end{equation*}
choices of $m_{\underline{j}}$ such that $|\widetilde{P}_1'(\bar{t})k_{j,1}-\widetilde{P}_2'(\bar{t})k_{j,2}|\lesssim\lambda    ^{\epsilon_3}$
\footnote{Recall that $\bar{t}=\bar{t}_\mathbf{m}$ with $\mathbf{m}=(m_1, m_2)$.}for some nonnegative constant $\mathfrak{c}$ (which is $0$ if $\mathfrak{r}_1\neq \mathfrak{r}_2$ and a positive integer if $\mathfrak{r}_1=\mathfrak{r}_2$). To prove this claim,  we may assume $j=1$ (while the case $j=2$ is similar). Let  $\mathbf{m}=(m_1,m_2)$, $\mathbf{m}'=(m_1',m_2')\in\mathcal{M}$ with $m_1=m_1'$. By the mean value theorem and \eqref{s2-11}, when $|m_{2,1}-m_{2,1}'|\geq 8$ we have
\begin{align*}
	\left|\bar{t}_\mathbf{m}-\bar{t}_{\mathbf{m}'}\right|&\gtrsim\left|\widetilde{P}_1(\bar{t}_\mathbf{m})-\widetilde{P}_1(\bar{t}_{\mathbf{m}'})\right|\geq\left|m_{2,1}-m_{2,1}'\right|\lambda^{-\gamma}-4\lambda^{-\gamma}\\
	&\gtrsim \left|m_{2,1}-m_{2,1}'\right|\lambda^{-\gamma}.
\end{align*}
We then have
\begin{align*}
	&\left|\left\{\frac{\widetilde{P}_1'}{\widetilde{P}_2'}\left(\bar{t}_{\mathbf{m}}\right)k_{1,1}-k_{1,2}\right\}-\left\{\frac{\widetilde{P}_1'}{\widetilde{P}_2'}\left(\bar{t}_{\mathbf{m}'}\right)k_{1,1}-k_{1,2}\right\}\right|\\
	=&\frac{|\widetilde{P}_1''\widetilde{P}_2'-\widetilde{P}_1'\widetilde{P}_2''|}{(\widetilde{P}_2')^2}\left(\widetilde{t}\right)\left|\bar{t}_{\mathbf{m}}-\bar{t}_{\mathbf{m}'}\right|\left|k_{1,1}\right|\\
	\gtrsim& 2^{-\mathfrak{c}|l|} \lambda^{-\gamma}\left|k_{1,1}\right|\left|m_{2,1}-m_{2,1}'\right|,
\end{align*}
where the constant $\mathfrak{c}\geq 0$ only depends on $P_1$, $P_2$ and the sign of $l$. Hence
\begin{equation*}
	\left|m_{2,1}-m_{2,1}'\right|\lesssim 1+2^{\mathfrak{c}|l|}\lambda^{\gamma+\epsilon_3}|k_{1,1}|^{-1}.
\end{equation*}
Notice that with $m_1$ and $m_{2,1}$ given, the choices of $m_{2,2}$ is $O(1)$. This fact was already used earlier in the estimate $\#\mathcal{M}\lesssim \lambda^{3\gamma}$. Thus we have proved the claim.

As a consequence, by splitting the sum into two parts depending on whether $|k_{j,j}|\leq \lambda^{\epsilon_4}$ or $>\lambda^{\epsilon_4}$ with a parameter $\epsilon_4\in(\epsilon_3,1-\gamma)$ (to be chosen later), we get for any fixed $s$ and $j\in\{1,2\}$ that
\begin{equation*}
	\sum_{\mathbf{m}\in\mathcal{M}}\!\!\!\!\!\sum_{\stackrel{k_j\in\mathbb{Z}^2}{|\widetilde{P}_1'(\bar{t})k_{j,1}-\widetilde{P}_2'(\bar{t})k_{j,2}|\lesssim\lambda^{\epsilon_3}}}\!\!\!\!\!\!\!\!\!\!\!\!\!\!\!\left|a_{j,m_j,k_j,s}\right|^2
	\lesssim \lambda^{\gamma}\sum_{m_j\in\mathcal{M}_j}\sum_{\stackrel{k_j\in\mathbb{Z}^2}{|k_{j,j}|\leq\lambda^{\epsilon_4}}}\left|a_{j,m_j,k_j,s}\right|^2+R_{s,j},
\end{equation*}
where
\begin{align*}
	R_{s,j}&=\left(1+2^{\mathfrak{c}|l|}\lambda^{\gamma+\epsilon_3-\epsilon_4}\right)\sum_{m_j\in\mathcal{M}_j}\sum_{k_j\in\mathbb{Z}^2}\left|a_{j,m_j,k_j,s}\right|^2\\
	&\lesssim2^{\mathfrak{c}|l|}\lambda^{\gamma+\epsilon_3-\epsilon_4}\sum_{m_j\in\mathcal{M}_j}F_{j,m_j}(s).
\end{align*}
Plugging this bound in the inequality \eqref{s3-26} yields for $j=1,2$ that
\begin{equation}\label{s2-12}
	\left\|\widetilde{T}_l(f_1,f_2)\right\|_1\lesssim\lambda^{\frac{1}{2}-\gamma}\|f_1\|_\infty\|f_2\|_\infty+2^{\frac{1}{4}\mathfrak{c}|l|}\lambda^{\frac{1}{2}\gamma+\frac{5}{4}\epsilon_3-\frac{1}{4}\epsilon_4}\Pi_1+\lambda^{\epsilon_3}\Pi_2^j	
\end{equation}
with
\begin{equation*}
	\Pi_1=G_1^{1/4}\cdot G_2^{1/4}
\end{equation*}
and
\begin{equation*}
	\Pi_2^j=G_{\underline{j}}^{1/4}\cdot\bigg(\int_{|s|\lesssim\lambda^{-\gamma}}\sum_{m_j\in\mathcal{M}_j}\sum_{\stackrel{k_j\in\mathbb{Z}^2}{|k_{j,j}|\leq\lambda^{\epsilon_4}}}\left|a_{j,m_j,k_j,s}\right|^2\,\mathrm{d}s\bigg)^{1/4}.
\end{equation*}
This \textit{basic estimate} \eqref{s2-12} will be used in the next subsection.


\subsection{Conclusion of the proof of Theorem \ref{thm3}}
In this subsection, we apply the following structural decomposition introduced by Christ,  Durcik and Roos in \cite{CDR21} and the basic estimate \eqref{s2-12}, reduce the problem to certain sublevel set estimates and then finish the proof of  Theorem \ref{thm3}.

\begin{lemma}[{Christ,  Durcik and Roos~\cite[Lemma 3.2]{CDR21}}]\label{l3-2}
	If $f\in L^2(\mathbb{R})$,  $\rho\in(0,1)$ and $R\geq1$, then there exists a decomposition
	\begin{equation*}
		f=f_\sharp+f_\flat
	\end{equation*}
	with the following properties.
	\begin{itemize}
		\item [(1)] One has
		\begin{equation*}
			\left|\widehat{f_\sharp}\right|,\left|\widehat{f_\flat}\right|\leq \left|\widehat{f}\,\right|.
		\end{equation*}
		Hence
		\begin{equation*}
			\left\|f_\sharp\right\|_2+\left\|f_\flat\right\|_2\lesssim\|f\|_2.
		\end{equation*}
		\item [(2)] The function $f_\sharp$ admits a decomposition
		\begin{equation*}
			f_\sharp(x)=\sum_{n=1}^{\mathcal{N}}h_n(x)e^{i\alpha_n x}
		\end{equation*}
		with each $\alpha_n\in \mathbb{R}$, $\mathcal{N}\lesssim\rho^{-1}$, and $ h_n $ a smooth function satisfying  that
		\begin{equation*}
			\supp\left(\widehat{h_n}\right)\subseteq[-R,R],
		\end{equation*}
		\begin{equation*}
			\|h_n\|_2\lesssim\|f\|_2
		\end{equation*}
		and
		\begin{equation*}
			\left\|\partial^N h_n\right\|_\infty\lesssim_N R^N\|f\|_\infty
		\end{equation*}	
		for all integers $N\geq 0$. Moreover, the support of $\widehat{f_\sharp}$ is contained in the support of $\widehat{f}$.
		\item[(3)] One has the bound\footnote{Here $\mathcal{D}_sf(x)=f(x+s)\overline{f(x)}$. }
		\begin{equation*}
			\int_\mathbb{R}\!\int_{|\xi|\leq R}\!\left|\widehat{\mathcal{D}_s f_\flat}(\xi)\right|^2\,\mathrm{d}\xi\mathrm{d}s\lesssim\rho\|f\|_2^4.
		\end{equation*}
	\end{itemize}
	All implicit constants do not depend on $R,\rho,f$.
\end{lemma}

Choose two Schwartz functions $\{\psi^{(j)}\}_{j=1}^2$ on $\mathbb{R}$ so that the Fourier transform of $\psi^{(j_*)}$ equals one on $\{\lambda\leq|\xi|\leq2\lambda\}$ and zero outside $\{\lambda/2\leq|\xi|\leq3\lambda\}$ and the Fourier transform of the other function, $\psi^{(3-j_*)}$, equals one on $ \{|\xi|\leq2\lambda\} $ and zero outside $ \{|\xi|\leq3\lambda\} $. By the assumptions \eqref{s3-4} and \eqref{s3-5}, we have
\begin{equation*}
	f_j=\psi^{(j)}\ast_j f_j \quad \textrm{for $j=1,2$.}
\end{equation*}
Let the bump function $\eta$ and $\eta_m$ be as defined at the beginning of Subsection \ref{sec3-1}. However we now use a modified definition
\begin{equation*}
	f_{j,m}=\psi^{(j)}\ast_j(\eta_m f_j).
\end{equation*}

Let $\tau=\gamma+\epsilon_4$ and $\delta'>0$ a parameter (to be determined later). We apply Lemma \ref{l3-2} (with parameters $R=\lambda^{\tau}$ and $\rho=\lambda^{-\delta'}$) to each function $x\mapsto f_{1,m}(x,y)$ with $y$ fixed. Then
\begin{equation*}
	f_{1,m}=f_{1,m,\flat}+f_{1,m,\sharp},
\end{equation*}
where
\begin{equation}\label{s3-35}
	f_{1,m,\sharp}(x,y)=\sum_{n=1}^{\mathcal{N}_1}h_{1,n,m}(x,y)e^{i\alpha_{n,m}(y)x}
\end{equation}
with $\mathcal{N}_1=O(\lambda^{\delta'})$, measurable real-valued functions $\alpha_{n,m}(y)$ such that $|\alpha_{n,m}|\lesssim\lambda$ (and $\asymp\lambda$ if $j_*$=1), measurable functions $h_{1,n,m}(x,y)$ smooth in $x$ such that $|\partial_x^N h_{1,n,m}|\lesssim_N\lambda^{\tau N}\|f_1\|_\infty$ (uniformly in $n,m$) and
\begin{equation}\label{s3-32}
	\int_\mathbb{R}\!\int_{\mathbb{R}^2}\!\mathbf{1}_{|\xi_1|\leq \lambda^{\tau}}\left|\widehat{\mathcal{D}_s^{(1)}f_{1,m,\flat}}(\xi)\right|^2\,\mathrm{d}\xi\mathrm{d}s\lesssim\lambda^{-\delta'-3\gamma}\|f_1\|_\infty^4.
\end{equation}
By a same operation on each function $y\mapsto f_{2,m}(x,y)$ with $x$ fixed, we obtain
\begin{equation*}
	f_{2,m}=f_{2,m,\flat}+f_{2,m,\sharp},
\end{equation*}
where
\begin{equation}\label{s3-36}
	f_{2,m,\sharp}(x,y)=\sum_{n=1}^{\mathcal{N}_2}h_{2,n,m}(x,y)e^{i\beta_{n,m}(x)y}
\end{equation}
with $\mathcal{N}_2=O(\lambda^{\delta'})$, measurable real-valued functions $\beta_{n,m}(x)$ such that $|\beta_{n,m}|\lesssim\lambda$ (and $\asymp\lambda$ if $j_*$=2), measurable functions $h_{2,n,m}(x,y)$ smooth in $y$ such that $|\partial_y^N h_{2,n,m}|\lesssim_N\lambda^{\tau N}\|f_2\|_\infty$ (uniformly in $n,m$) and
\begin{equation*}
	\int_\mathbb{R}\!\int_{\mathbb{R}^2}\!\mathbf{1}_{|\xi_2|\leq \lambda^{\tau}}\left|\widehat{\mathcal{D}_s^{(2)}f_{2,m,\flat}}(\xi)\right|^2\,\mathrm{d}\xi\mathrm{d}s\lesssim\lambda^{-\delta'-3\gamma}\|f_2\|_\infty^4.
\end{equation*}

In order to apply the basic estimate one can spatially localize the functions $f_{j,m,\flat}$ and $f_{j,m,\sharp}$. Denote by $\widetilde{\eta}$ a smooth function that equals one on the support of $\eta$ and has a slightly larger support than $\eta$, and $\widetilde{\eta}_m(x,y)=\widetilde{\eta}(\lambda^\gamma(x,y)-m)$. Write
\begin{align*}
	f_j&=\sum_{m\in\mathbb{Z}^2}\widetilde{\eta}_m\eta_mf_j\\
	&=\sum_{m\in\mathbb{Z}^2}\widetilde{\eta}_mf_{j,m}+\sum_{m\in\mathbb{Z}^2}\widetilde{\eta}_m\left[\eta_m\left(\psi^{(j)}\ast_jf_j\right)-\psi^{(j)}\ast_j\left(\eta_mf_j\right)\right]\\
	&=:f_{j,\flat}+f_{j,\sharp}+f_{j,\text{err}}
\end{align*}
with
\begin{align}
	f_{j,\flat}&=\sum_{m\in\mathbb{Z}^2}\widetilde{\eta}_mf_{j,m,\flat}, \nonumber\\
	f_{j,\sharp}&=\sum_{m\in\mathbb{Z}^2}\!\widetilde{\eta}_mf_{j,m,\sharp},\label{s3-15}\\
	f_{j,\text{err}}&=\sum_{m\in\mathbb{Z}^2}\widetilde{\eta}_m\left[\eta_m\left(\psi^{(j)}\ast_jf_j\right)-\psi^{(j)}\ast_j\left(\eta_mf_j\right)\right].\nonumber
\end{align}
It is not hard to show that
\begin{align}
	\|f_{j,\flat}\|_\infty, \|f_{j,\sharp}\|_\infty&\lesssim\lambda^{\delta'}\|f_j\|_\infty,\label{s3-38}\\
	\|f_{j,\text{err}}\|_\infty&\lesssim\lambda^{\gamma-1}\|f_j\|_\infty.\label{s3-39}
\end{align}
With the above decomposition of $f_j$, we have
\begin{equation*}
	\widetilde{T}_l(f_1,f_2)=\widetilde{T}_\sharp+\widetilde{T}_\flat+\widetilde{T}_{\text{err}},
\end{equation*}
where
\begin{align*}
	\widetilde{T}_\sharp&=\widetilde{T}_l(f_{1,\sharp},f_{2,\sharp}),\\
	\widetilde{T}_\flat&=\widetilde{T}_l(f_{1,\flat},f_2)+\widetilde{T}_l(f_{1,\sharp},f_{2,\flat}),\\
	\widetilde{T}_{\text{err}}&=\widetilde{T}_l(f_{1,\sharp},f_{2,\text{err}})+\widetilde{T}_l(f_{1,\text{err}},f_2).
\end{align*}

For $\widetilde{T}_{\text{err}}$, by \eqref{s3-38} and \eqref{s3-39}, we have
\begin{align*}
	\|\widetilde{T}_{\text{err}}\|_1&\leq\|\widetilde{T}_l(f_{1,\sharp},f_{2,\text{err}})\|_1+\|\widetilde{T}_l(f_{1,\text{err}},f_2)\|_1\\
	&\lesssim\|f_{1,\sharp}\|_\infty\|f_{2,\text{err}}\|_\infty+\|f_{1,\text{err}}\|_\infty\|f_2\|_\infty\\
	&\lesssim \lambda^{\gamma-1+\delta'}\|f_1\|_\infty\|f_2\|_\infty.
\end{align*}

For $\widetilde{T}_\flat$ we first apply to $\widetilde{T}_l(f_{1,\flat},f_2)$  the basic estimate \eqref{s2-12} with $f_1$ and $f_{1,m}$ there replaced by $f_{1,\flat}$ and $\widetilde{\eta}_mf_{1,m,\flat}$ respectively. Then we get that
\begin{align*}
	\left\|\widetilde{T}_l(f_{1,\flat},f_2)\right\|_1\lesssim&\lambda^{\frac{1}{2}-\gamma}\|f_{1,\flat}\|_\infty\|f_2\|_\infty+
	2^{\frac{1}{4}\mathfrak{c}|l|}\lambda^{\frac{1}{2}\gamma+\frac{5}{4}\epsilon_3-\frac{1}{4}\epsilon_4}G_1^{1/4}\cdot G_2^{1/4}+\\
	&\lambda^{\epsilon_3} G_2^{1/4}\cdot\bigg(\int_{|s|\lesssim\lambda^{-\gamma}}\!\sum_{m_1\in\mathcal{M}_1}\!\!\!\sum_{\stackrel{k_1\in\mathbb{Z}^2}{|k_{1,1}|\leq\lambda^{\epsilon_4}}}\!\!\!\!\!\!\left|a_{1,m_1,k_1,s}\right|^2\,\mathrm{d}s\bigg)^{1/4}.
\end{align*}
Using definitions of $G_j$ and $F_{j,m_j}$ and Lemma \ref{l3-2}, we get
\begin{equation*}
	\int_{|s|\lesssim\lambda^{-\gamma}}\!F_{j,m_j}(s)\,\mathrm{d}s\lesssim\lambda^{-\gamma}\|f_j\|_\infty^4
\end{equation*}
and therefore
\begin{equation*}
	G_j\lesssim\lambda^{-\gamma}\|f_j\|_\infty^4.
\end{equation*}
Moreover, by using the fast decay of the Fourier transform of $\mathcal{D}_{\lambda^{\gamma}s}^{(1)}\widetilde{\eta}$ and the bound \eqref{s3-32}, we get for any $N\in\mathbb{N}$ that
\begin{equation*} 
	\int_{|s|\lesssim\lambda^{-\gamma}}\!\sum_{m_1\in\mathcal{M}_1}\!\!\!\sum_{\stackrel{k_1\in\mathbb{Z}^2}{|k_{1,1}|\leq\lambda^{\epsilon_4}}}\!\!\!\!\!\!|a_{1,m_1,k_1,s}|^2\,\mathrm{d}s\lesssim_N \left(\lambda^{\gamma-\delta'}+\lambda^{\gamma-\epsilon_4N}\right)\|f_1\|_\infty^4.
\end{equation*}
Thus
\begin{equation*}
	\left\|\widetilde{T}_l(f_{1,\flat},f_2)\right\|_1\lesssim_{\epsilon_4}2^{\frac{1}{4}\mathfrak{c}|l|}\left(\lambda^{\frac{1}{2}-\gamma+\delta'}+\lambda^{\frac{5}{4}\epsilon_3-\frac{1}{4}\epsilon_4}+\lambda^{-\frac{1}{4}\delta'+\epsilon_3}\right)\|f_{1}\|_\infty\|f_2\|_\infty.
\end{equation*}
Since the treatment of the second term of $\widetilde{T}_\flat$ is similar, we omit the details. We then get
\begin{equation*}
	\left\|\widetilde{T}_\flat\right\|_1\lesssim_{\epsilon_4}2^{\frac{1}{4}\mathfrak{c}|l|}\left(\lambda^{\frac{1}{2}-\gamma+2\delta'}+\lambda^{\frac{5}{4}\epsilon_3-\frac{1}{4}\epsilon_4}+\lambda^{-\frac{1}{4}\delta'+\epsilon_3}\right)\|f_{1}\|_\infty\|f_2\|_\infty.
\end{equation*}

It remains to estimate the term $\|\widetilde{T}_\sharp\|_1$. One can follow the method used in \cite{CDR21} to reduce such an estimation to certain sublevel set estimates. We sketch the proof for the sake of completeness.

By using \eqref{s3-1},  \eqref{s3-35}, \eqref{s3-36} and \eqref{s3-15}, we observe that $\|\widetilde{T}_\sharp\|_1$ is bounded by a sum of $O(\lambda^{2\delta'})$ terms of the form
\begin{equation}
\sum_{\mathbf{m}\in\widetilde{\mathcal{M}}}\iint\!\left|\int\, e^{i\left(\alpha_{m_1}(y)\widetilde{P}_1(t)+\beta_{m_2}(x)\widetilde{P}_2(t)\right)}H_{\mathbf{m}}(x,y,t)\,\mathrm{d}t\right|\mathrm{d}x\mathrm{d}y,              \label{s3-46}
\end{equation}
where $\alpha_{m_1}$ and $\beta_{m_2}$ are measurable real-valued functions such that $|\alpha_{m_1}|$, $|\beta_{m_2}|\lesssim \lambda$ and $|\alpha_{m_1}|\asymp\lambda$ if $j_*=1$ while $|\beta_{m_2}|\asymp\lambda$ if $j_*=2$, and $H_{\mathbf{m}}(x,y,t)$ is a measurable function of the form
\begin{equation*}
	H_{\mathbf{m}}(x,y,t)=\left(\widetilde{\eta}_{m_1}h_{1,m_1}\right)\left(x+\widetilde{P}_1(t),y\right)\left(\widetilde{\eta}_{m_2}h_{2,m_2}\right)\left(x,y+\widetilde{P}_2(t)\right)\zeta(x,y,t)
\end{equation*}
with $\widetilde{\eta}_m$ supported in a cube $\widetilde{Q}_m$ of side length $\asymp\lambda^{-\gamma}$ centered at $\lambda^{-\gamma}m$ and measurable functions $h_{1,m}$ and $h_{2,m}$ (smooth in $x$ and $y$ respectively) satisfying $\|\partial_j^N h_{j,m}\|_\infty\lesssim_N\lambda^{\tau N}\|f_j\|_\infty$ uniformly in $m$. Notice that $H_{\mathbf{m}}$ is smooth in $t$ such that
\begin{equation*}
	\left\|\partial_t^NH_{\mathbf{m}}\right\|_\infty\lesssim \lambda^{\tau N}\|f_1\|_\infty\|f_2\|_\infty.
\end{equation*}
The index set $\widetilde{\mathcal{M}}$ is the set of $\mathbf{m}=(m_1, m_2)\in(\mathbb{Z}^2)^2$ such that $\widetilde{Q}_{m_1}\in\mathcal{Q}_1$, $\widetilde{Q}_{m_2}\in\mathcal{Q}_2$ and the summand in  \eqref{s3-46} is nonzero, where  $\mathcal{Q}_1$, $\mathcal{Q}_2\in\{\mathcal{Q}\}$ with each $\mathcal{Q}$ consisting of pairwise disjoint cubes. Notice that
\begin{equation*}
	\#\widetilde{\mathcal{M}}\lesssim\lambda^{3\gamma}.
\end{equation*}

For each $\mathbf{m}\in\widetilde{\mathcal{M}}$, fix a point $(\bar{x}_{\mathbf{m}},\bar{y}_{\mathbf{m}},\bar{t}_{\mathbf{m}})$ in the support of $H_\mathbf{m}$. Let $\rho>0$ be a small parameter to be determined later. For each $(x,y,t)$ in the support of $H_\mathbf{m}$, if
\begin{equation*}
	\left|\alpha_{m_1}(y)\widetilde{P}_1'(\bar{t}_{\mathbf{m}})+\beta_{m_2}(x)\widetilde{P}_2'(\bar{t}_{\mathbf{m}})\right|\geq\lambda^{\tau+\rho},
\end{equation*}
then, by the mean value theorem, we have
\begin{equation*}
	\left|\alpha_{m_1}(y)\widetilde{P}_1'(t)+\beta_{m_2}(x)\widetilde{P}_2'(t)\right|\gtrsim\lambda^{\tau+\rho}
\end{equation*}
and, by integration by parts, we have
\begin{equation*}
	\left|\int\, e^{i[\alpha_{m_1}(y)\widetilde{P}_1(t)+\beta_{m_2}(x)\widetilde{P}_2(t)]}H_{\mathbf{m}}(x,y,t)\,\mathrm{d}t\right|\lesssim_N\lambda^{-\rho N}\|f_1\|_\infty\|f_2\|_\infty
\end{equation*}
for any $N\in \mathbb{N}$. Therefore
\begin{align}
\|\widetilde{T}_\sharp\|_1\lesssim   &\lambda^{-N}\|f_1\|_\infty\|f_2\|_\infty+ \nonumber\\	
 &\lambda^{2\delta'}\sum_{\mathbf{m}\in\widetilde{\mathcal{M}}}\int_K\!\!\left|H_\mathbf{m}(x,y,t)\right|\mathbf{1}_{|\alpha_{m_1}(y)\widetilde{P}_1'(\bar{t}_{\mathbf{m}})+\beta_{m_2}(x)\widetilde{P}_2'(\bar{t}_{\mathbf{m}})|\leq\lambda^{\tau+\rho}}\,\mathrm{d}x\mathrm{d}y\mathrm{d}t,   \label{s3-47}
\end{align}
where $K\subset \mathbb{R}^2\times[1/2,2]$ denotes the compact support of $\zeta$.

Let $E_{\mathbf{m}}\subseteq\mathbb{R}^3$ be the set of $(x,y,t)$ satisfying $(x+\widetilde{P}_1(t),y)\in \widetilde{Q}_{m_1}$, $(x,y+\widetilde{P}_2(t))\in \widetilde{Q}_{m_2}$ and $|\alpha_{m_1}(y)\widetilde{P}_1'(t)+\beta_{m_2}(x)\widetilde{P}_2'(t)|\leq 2\lambda^{\tau+\rho}$. Recall that $\mathcal{Q}_1,\mathcal{Q}_2$ are both collections of pairwise disjoint cubes. Then for each $1\leq j\leq 2$ and each fixed $(x,y)\in\mathbb{R}^2$, there exists at most one $m\in\mathbb{Z}^2$ with $(x,y)\in\widetilde{Q}_m\in\mathcal{Q}_j$. Therefore one can define two measurable functions $\mathfrak{m}_j:\mathbb{R}^2\rightarrow\mathbb{Z}^2$ by
\begin{equation*}
	\mathfrak{m}_j(x,y)=\begin{cases}
		m,& \text{  if $(x,y)\in\widetilde{Q}_m\in\mathcal{Q}_j$,}\\
		0,& \text{if no such $m$ exists,}
	\end{cases}
\end{equation*}
for $j=1,2$. Thus for $(x,y,t)\in K$,
\begin{equation*}
	\sum_{\mathbf{m}\in\widetilde{\mathcal{M}}}\mathbf{1}_{E_{\mathbf{m}}}(x,y,t)\leq\mathbf{1}_{|\widetilde{P}_1'(t)\widetilde{\alpha}(x+\widetilde{P}_1(t),y)-\widetilde{P}_2'(t)\widetilde{\beta}(x,y+\widetilde{P}_2(t))|\leq \varepsilon}(x,y,t)
\end{equation*}
with $\varepsilon=2\lambda^{\tau+\rho-1}$ and
\begin{equation*}
	\widetilde{\alpha}(x,y)=\lambda^{-1}\alpha_{\mathfrak{m}_1(x,y)}(y),\quad\widetilde{\beta}(x,y)=-\lambda^{-1}\beta_{\mathfrak{m}_2(x,y)}(x)
\end{equation*}
two measurable real-valued functions satisfying either $|\widetilde{\alpha}|\asymp 1$ or $|\widetilde{\beta}|\asymp1$. Hence
\begin{align*}
	\eqref{s3-47}&\lesssim \lambda^{2\delta'}\|f_1\|_\infty\|f_2\|_\infty\int_K\!\sum_{\mathbf{m}\in\widetilde{\mathcal{M}}}\mathbf{1}_{E_{\mathbf{m}}}(x,y,t)\,\mathrm{d}x\mathrm{d}y\mathrm{d}t\\
	&\leq\lambda^{2\delta'}\|f_1\|_\infty\|f_2\|_\infty \int_K\!\mathbf{1}_{|\widetilde{P}_1'(t)\widetilde{\alpha}(x+\widetilde{P}_1(t),y)-\widetilde{P}_2'(t)\widetilde{\beta}(x,y+\widetilde{P}_2(t))|\leq \varepsilon}\,\mathrm{d}x\mathrm{d}y\mathrm{d}t\\
	&\leq\lambda^{2\delta'}\|f_1\|_\infty\|f_2\|_\infty|E(\varepsilon)|
\end{align*}
with a sublevel set
\begin{equation*}
E(\varepsilon)=\!\left\{(x,y,t)\in \!K:\left|\widetilde{P}_1'(t)\widetilde{\alpha}\left(x+\widetilde{P}_1(t),y\right)\!-\widetilde{P}_2'(t)\widetilde{\beta}\left(x,y+\widetilde{P}_2(t)\right)\right|\!\leq \varepsilon\right\}.
\end{equation*}
Therefore the problem is reduced to the estimate of $|E(\varepsilon)|$.

By Proposition \ref{l4-2} (or Proposition \ref{l4-1} if $\sigma_1\neq \sigma_2$ when $l>\Gamma$), there exist constants $\kappa>0$ and $c\geq0$ (which is $0$ if $\mathfrak{r}_1\neq \mathfrak{r}_2$ and positive if $\mathfrak{r}_1=\mathfrak{r}_2$) such that
\begin{equation*}
	|E(\varepsilon)|\lesssim2^{c|l|}\varepsilon^{\kappa}\lesssim2^{c|l|}\lambda^{\kappa(\tau+\rho-1)}.
\end{equation*}
Thus
\begin{equation*}
\left\|\widetilde{T}_\sharp\right\|_1\lesssim 2^{c|l|}\lambda^{\kappa(\tau+\rho-1)+2\delta'}\|f_1\|_\infty\|f_2\|_\infty
\end{equation*}
and collecting bounds for  $\|\widetilde{T}_{\text{err}}\|_1$, $\|\widetilde{T}_\flat\|_1$ and $\|\widetilde{T}_\sharp\|_1$ yields
\begin{align*}
	\left\|\widetilde{T}_l(f_1,f_2)\right\|_1\lesssim2^{\mathfrak{b}_2|l|}\bigg(\lambda^{\gamma-1+\delta'}&+\lambda^{\frac{1}{2}-\gamma+2\delta'}+\lambda^{-\frac{1}{4}\epsilon_4+\frac{5}{4}\epsilon_3}\\
	&+\lambda^{-\frac{1}{4}\delta'+\epsilon_3}+\lambda^{\kappa(\tau+\rho-1)+2\delta'}\bigg)\|f_1\|_\infty\|f_2\|_\infty
\end{align*}
with some constant $\mathfrak{b}_2\geq0$ (which is $0$ if $\mathfrak{r}_1\neq \mathfrak{r}_2$ and positive if $\mathfrak{r}_1=\mathfrak{r}_2$).
By choosing proper parameters (for example,
\begin{equation*}
\gamma=\frac{22\kappa+9}{26\kappa+18},\ \delta'=\epsilon_4=\frac{2\kappa}{13\kappa+9}
\end{equation*}
and sufficiently small $\rho$, $\epsilon_3$), we get for some constant $\sigma>0$ that
\begin{equation*}
	\left\|\widetilde{T}_l(f_1,f_2)\right\|_1\lesssim2^{\mathfrak{b}_2|l|}\lambda^{-\sigma}\|f_1\|_\infty\|f_2\|_\infty,
\end{equation*}
which is the desired \eqref{s3-3}. \qed


\section{Sublevel set estimates} \label{sec3}

In this section we prove two $\varepsilon$-bounds (with explicit exponents) for certain sublevel sets involving two polynomials. Our results generalize Christ, Durcik and Roos' \cite[Lemma 3.3]{CDR21}.

Let $P_1$ and $P_2$ be two polynomials denoted by \eqref{s1-1} and \eqref{s1-2}. Recall that for $j=1,2$ we define
\begin{equation*}
\widetilde{P}_j(t)=\widetilde{P}_{j,l}(t)=2^{\sigma_j l}P_j(2^{-l}t) \textrm{ if $l>\Gamma$},
\end{equation*}
and
\begin{equation*}
\widetilde{P}_j(t)=\widetilde{P}_{j,l}(t)=2^{d_j l}P_j(2^{-l}t) \textrm{ if $l<-\Gamma$}.
\end{equation*}
In particular the case $P_1(t)=t$ and $P_2(t)=t^2$ was studied in \cite{CDR21}.

\begin{proposition}\label{l4-2}
Let $P_1$ and $P_2$ be linearly independent polynomials  with zero constant term, $I=[1/2,2]$, $K\subset\mathbb{R}^2\times I$ a compact set and $\alpha, \beta$: $\mathbb{R}^2\rightarrow\mathbb{R}$  measurable functions with either $|\alpha|\asymp 1$ or $|\beta|\asymp1$. There exists a nonnegative constant $c$ (which is $0$ if $l<-\Gamma$ and $d_1\neq d_2$ or if $l>\Gamma$ and $\sigma_1\neq \sigma_2$, is a positive integer depending only on $P_1$ and $P_2$ otherwise) such that if $|l|>\Gamma$ is sufficiently large then
\begin{equation}
\begin{split}
		&\left|\left\lbrace(x,y,t)\in K:\left|\widetilde{P}_1'(t)\alpha(x+\widetilde{P}_1(t),y)-\widetilde{P}_2'(t)\beta(x,y+\widetilde{P}_2(t))\right|\leq \varepsilon\right\rbrace\right|\\
\lesssim & 2^{c|l|}\varepsilon^{\frac{1}{7(8(d_1+d_2)-17)}}
\end{split}\label{s5-11}
\end{equation}
for all $\varepsilon\in (0,1]$. The implicit constant depends only on $K$, $P_1$ and $P_2$, but not on measurable functions $\alpha$ and $\beta$.	
\end{proposition}

\begin{proof}
We may assume that $|\alpha|\asymp 1$ (since the case $|\beta|\asymp 1$ can be handled similarly) and $|\beta|\lesssim 1$ (otherwise the sublevel set in \eqref{s5-11} is empty). We may assume $l>\Gamma$ while the case $l<-\Gamma$ can be handled similarly.

By using the fact $|\widetilde{P}_1(t)|\asymp 1$ if $l$ is sufficiently large and changing variables
\begin{equation}
(x,y,t)\mapsto(x-\widetilde{P}_1(t),y,t), \label{s5-2}
\end{equation}
 we only need to study the set
\begin{equation}
\mathcal{E}=\{(\mathbf{z},t)\in K': |\alpha(\mathbf{z})-H(t)\beta(\mathbf{z}+\mathbf{p}(t))|\leq\varepsilon\}, \label{s4-2}
\end{equation}
where $\mathbf{z}=(x,y)$, $\mathbf{p}(t)=(-\widetilde{P}_1(t),\widetilde{P}_2(t))$, $H(t)=\widetilde{P}_2'(t)/\widetilde{P}_1'(t)$ and the compact set $K'\subset\mathbb{R}^2\times I$ is the image of $K$ under the mapping \eqref{s5-2}. It suffices to show that
\begin{equation*}
|\mathcal{E}|\lesssim 2^{c|l|}\varepsilon^{\frac{1}{7(8(d_1+d_2)-17)}}
\end{equation*}
for some constant $c\geq 0$. We may further assume that $K'=[0,1]^2\times I$ (which can be achieved by covering $K'$ (essentially $K$) by finitely many rectangles and translating each rectangle to $[0,1]^2\times I$) and $|\mathcal{E}|>0$.

As in \cite{CDR21}, we first claim that there exist a point $\bar{\mathbf{z}}\in[0,1]^2$ and a measurable set $\mathcal{A}\subset I^3$ so that $|\mathcal{E}|\lesssim |\mathcal{A}|^{1/7}$ and for every $(t_1,t_2,t_3)\in\mathcal{A}$,
\begin{equation}\label{s5-3}
	\begin{cases}
		(\bar{\mathbf{z}},t_1)\in\mathcal{E},\\
		(\bar{\mathbf{z}}+\mathbf{p}(t_1)-\mathbf{p}(t_2),t_2)\in\mathcal{E},\\
		(\bar{\mathbf{z}}+\mathbf{p}(t_1)-\mathbf{p}(t_2),t_3)\in\mathcal{E},
	\end{cases}
\end{equation}
namely
\begin{equation}\label{s4-3}
	\begin{cases}
		|\alpha(\bar{\mathbf{z}})-H(t_1)\beta(\bar{\mathbf{z}}+\mathbf{p}(t_1))|\leq\varepsilon,\\
		|\alpha(\bar{\mathbf{z}}+\mathbf{p}(t_1)-\mathbf{p}(t_2))-H(t_2)\beta(\bar{\mathbf{z}}+\mathbf{p}(t_1))|\leq\varepsilon,\\
		|\alpha(\bar{\mathbf{z}}+\mathbf{p}(t_1)-\mathbf{p}(t_2))-H(t_3)\beta(\bar{\mathbf{z}}+\mathbf{p}(t_1)-\mathbf{p}(t_2)+\mathbf{p}(t_3))|\leq\varepsilon.
	\end{cases}
\end{equation}

We apply the method used in \cite{CDR21} to prove this claim. To see this, define 	
\begin{equation*}
\mathcal{E}_0'=\{\mathbf{z}\in\mathbb{R}^2:|\{t\in I:(\mathbf{z},t)\in\mathcal{E}\}|\geq |\mathcal{E}|/2\} \subset [0,1]^2.
\end{equation*}
Then $|\mathcal{E}_0'|\geq |\mathcal{E}|/4$. This is clear because, by the definition of $\mathcal{E}_0'$,
\begin{equation*}
|\mathcal{E}|=\int_{\mathcal{E}_0'}\!\int_{I}\!\mathbf{1}_{\mathcal{E}}(\mathbf{z},t)\,\mathrm{d}t\mathrm{d}\mathbf{z}+
\int_{[0,1]^2\setminus\mathcal{E}_0'}\!\int_{I}\!\mathbf{1}_{\mathcal{E}}(\mathbf{z},t)\,\mathrm{d}t\mathrm{d}\mathbf{z}
	\leq 2|\mathcal{E}_0'|+\frac{1}{2}|\mathcal{E}|.
\end{equation*}
Define
\begin{equation*}
	\mathcal{E}_1=\{(\mathbf{z},t)\in\mathcal{E}:\mathbf{z}\in\mathcal{E}_0'\}.
\end{equation*}
Then $|\mathcal{E}_1|\geq |\mathcal{E}|^2/8$. Indeed,  by the definition of $\mathcal{E}_0'$,
\begin{equation*}
|\mathcal{E}_1|=\int_{[0,1]^2}\!\int_{I}\!\mathbf{1}_{\mathcal{E}}(\mathbf{z},t)\mathbf{1}_{\mathcal{E}_0'}(\mathbf{z})\,\mathrm{d}t\mathrm{d}\mathbf{z}
\geq\frac{1}{2}|\mathcal{E}||\mathcal{E}_0'|\geq\frac{1}{8}|\mathcal{E}|^2.
\end{equation*}
Define
\begin{equation*}
	\mathcal{E}_1'=\{\mathbf{z}\in\mathbb{R}^2:|\{t\in I:(\mathbf{z}-\mathbf{p}(t),t)\in\mathcal{E}_1\}|\geq c_1|\mathcal{E}_1|\}
\end{equation*}
for a sufficiently small constant $c_1=c_1(P_1, P_2)$ (to be determined below). Then $|\mathcal{E}_1'|\geq |\mathcal{E}|^2/32$. Indeed, by changing variables $\mathbf{z}\mapsto \mathbf{z}-\mathbf{p}(t)$ and the definition of $\mathcal{E}_1'$,
\begin{align*}
|\mathcal{E}_1|&=\int_{\mathcal{E}_1'}\!\int_{I}\!\mathbf{1}_{\mathcal{E}_1}(\mathbf{z}-\mathbf{p}(t),t)\,\mathrm{d}t\mathrm{d}\mathbf{z}+\int_{\mathbb{R}^2\setminus \mathcal{E}_1'}\!\int_{I}\!\mathbf{1}_{\mathcal{E}_1}(\mathbf{z}-\mathbf{p}(t),t)\,\mathrm{d}t\mathrm{d}\mathbf{z}\\
	&\leq 2|\mathcal{E}_1'|+c_1C_{P_1, P_2}|\mathcal{E}_1|
\end{align*}
for some positive constant $C_{P_1, P_2}$. If we choose $c_1=1/(2C_{P_1, P_2})$, then $|\mathcal{E}_1'|\geq |\mathcal{E}_1|/4 \geq |\mathcal{E}|^2/32$.
Define
\begin{equation*}
	\mathcal{E}_2=\{(\mathbf{z},t)\in \mathbb{R}^2\times I:\mathbf{z}\in\mathcal{E}_1'\text{ and }(\mathbf{z}-\mathbf{p}(t),t)\in\mathcal{E}_1\}.
\end{equation*}
Then $|\mathcal{E}_2|\geq c_1|\mathcal{E}_1||\mathcal{E}_1'|\geq 2^{-8}c_1|\mathcal{E}|^4$. Define
\begin{equation*}
\mathcal{E}_2'=\{\mathbf{z}\in \mathbb{R}^2:|\{t\in I:(\mathbf{z}+\mathbf{p}(t),t)\in\mathcal{E}_2\}|\geq |\mathcal{E}_2|/2\}.
\end{equation*}
Then $\mathcal{E}_2'\subset [0,1]^2$ and  $|\mathcal{E}_2'|\geq 2^{-10}c_1|\mathcal{E}|^4>0$. Indeed, by changing variables $\mathbf{z}\mapsto \mathbf{z}+\mathbf{p}(t)$ and definitions of $\mathcal{E}_2$ and $\mathcal{E}_2'$,
\begin{align*}	|\mathcal{E}_2|&=\int_{\mathcal{E}_2'}\!\int_{I}\!\mathbf{1}_{\mathcal{E}_2}(\mathbf{z}+\mathbf{p}(t),t)\,\mathrm{d}t\mathrm{d}\mathbf{z}+\int_{\mathbb{R}^2\setminus \mathcal{E}_2'}\!\int_{I}\!\mathbf{1}_{\mathcal{E}_2}(\mathbf{z}+\mathbf{p}(t),t)\,\mathrm{d}t\mathrm{d}\mathbf{z}\\
	&\leq 2|\mathcal{E}_2'|+\int_{[0,1]^2\setminus \mathcal{E}_2'}\!\int_{I}\!\mathbf{1}_{\mathcal{E}_2}(\mathbf{z}+\mathbf{p}(t),t)\,\mathrm{d}t\mathrm{d}\mathbf{z} \leq 2|\mathcal{E}_2'|+\frac{1}{2}|\mathcal{E}_2|.
\end{align*}
Hence $|\mathcal{E}_2'|\geq |\mathcal{E}_2|/4\geq 2^{-10}c_1|\mathcal{E}|^4$.

We now fix an arbitrary point $\bar{\mathbf{z}}\in\mathcal{E}_2'$ and denote
\begin{equation*}
	U=\{t\in I:(\bar{\mathbf{z}}+\mathbf{p}(t),t)\in\mathcal{E}_2\},
\end{equation*}
\begin{equation*}
	U_{t_1}=\{t\in I:(\bar{\mathbf{z}}+\mathbf{p}(t_1)-\mathbf{p}(t),t)\in\mathcal{E}_1\}  \ \textrm{for each $t_1\in U$},
\end{equation*}
\begin{equation*}
	U_{t_1,t_2}=\{t\in I:(\bar{\mathbf{z}}+\mathbf{p}(t_1)-\mathbf{p}(t_2),t)\in\mathcal{E}\} \ \textrm{for each $t_1\in U$ and $t_2\in U_{t_1}$}.
\end{equation*}
We remark that these sets are well-defined since $U$ and $U_{t_1}$ are nonempty sets. In fact it is easy to find that
$|U|\gtrsim|\mathcal{E}|^4$, $|U_{t_1}|\gtrsim|\mathcal{E}|^2$ and $|U_{t_1,t_2}|\gtrsim|\mathcal{E}|$. Define
\begin{equation*}
	\mathcal{A}=\{(t_1,t_2,t_3)\in I^3 : t_1\in U, t_2\in U_{t_1} \textrm{ and } t_3\in U_{t_1,t_2}\}.
\end{equation*}
Then $|\mathcal{A}|\gtrsim |\mathcal{E}|^7$. The properties \eqref{s5-3} follow easily from definitions of $\mathcal{E}_1$, $\mathcal{E}_2$,  $U$, $U_{t_1}$ and $U_{t_1,t_2}$. This completes the proof of the claim.

We next prove an upper bound of $|\mathcal{A}|$.  We define a function
\begin{equation*}
F(\mathbf{t})=\alpha(\bar{\mathbf{z}})H(t_1)^{-1}H(t_2)H(t_3)^{-1}-\beta(\bar{\mathbf{z}}+\mathbf{p}(t_1)-\mathbf{p}(t_2)+\mathbf{p}(t_3))
\end{equation*}
with $\mathbf{t}=(t_1,t_2,t_3)\in \mathbb{R}^3$. By \eqref{s4-3}, we find that
\begin{equation*}
|F(\mathbf{t})|\lesssim\varepsilon \textrm{ for every $\mathbf{t}\in\mathcal{A}$},
\end{equation*}
since $|\alpha(\bar{\mathbf{z}})|\asymp1$ and $|H(t_i)|\asymp 1$. Hence $\mathcal{A}\subset \Omega(C\varepsilon)$ with
\begin{equation}
\Omega(\varepsilon):=\{\mathbf{t}\in I^3:|F(\mathbf{t})|\leq\varepsilon\} \textrm{ for any $\varepsilon>0$}.\label{s5-13}
\end{equation}
It suffices to estimate the measure of the sublevel set $\Omega(\varepsilon)$.

As a preparation of the estimation, we introduce two polynomials from $[1/4,4]^3$ to $\mathbb{R}$ (which will arise naturally later) as follows
\begin{equation}
\begin{split}
Q(t_1,t_2,t_3)=&\left(\widetilde{P}_1'\widetilde{P}_2'\right)(t_2)\left(\widetilde{P}_1'\widetilde{P}_2'\right)(t_3)\left(\widetilde{P}_1'(t_2)\widetilde{P}_2'(t_3)-\widetilde{P}_1'(t_3)\widetilde{P}_2'(t_2)\right)\cdot\\
	           &\left(\widetilde{P}_2''\widetilde{P}_1'(t_1)-\widetilde{P}_1''\widetilde{P}_2'(t_1)\right)\label{s5-22}
\end{split}
\end{equation}
and
\begin{equation*}
\widetilde{Q}(\mathbf{t})=2^{(\mathfrak{d}-4(\sigma_1+\sigma_2)+10)l}\frac{Q(t_1,t_2,t_3)+Q(t_2,t_3,t_1)+Q(t_3,t_1,t_2)}{t_2-t_3}
\end{equation*}
with $\mathfrak{d}$ denoting the smallest degree of any term in $\widetilde{Q}(\mathbf{t})$ with nonzero coefficient. We observe that $\widetilde{Q}$ is indeed a polynomial because its numerator (treated as a function of $t_2$) has a zero at $t_3$. After investigating the expansion of its numerator, we observe that
\begin{equation*}
\max\{2, \ 4(\sigma_1+\sigma_2)-10\}\leq \mathfrak{d}\leq 4(d_1+d_2)-10
\end{equation*}
and that $\widetilde{Q}$ is equal to terms of degree $\mathfrak{d}$ (which are independent of $l$ as a result of the $2^l$ factor in the definition) plus terms of size $O(2^{-l})$. Hence $|\nabla \widetilde{Q}|\lesssim 1$ and there exists a multi-index $\beta$ with $|\beta|=\mathfrak{d}$ such that $|D^{\beta}\widetilde{Q}|\gtrsim 1$ if $l$ is sufficiently large.

We now split $\Omega({\varepsilon})$ into three parts
\begin{equation*}
\Omega({\varepsilon})\subseteq \Omega_1 \cup \Omega_2 \cup \Omega_3,
\end{equation*}
where we define
\begin{equation*}
\Omega_1=\{\mathbf{t}\in I^3: |\widetilde{Q}(\mathbf{t})|>\triangle_2, \ |t_2-t_3|>\triangle_1 \},
\end{equation*}
\begin{equation*}
\Omega_2=\{\mathbf{t}\in I^3: |t_2-t_3|\leq \triangle_1\}
\end{equation*}
and
\begin{equation*}
\Omega_3=\{\mathbf{t}\in I^3: |\widetilde{Q}(\mathbf{t})|\leq \triangle_2 \}
\end{equation*}
with
\begin{equation}
\triangle_1=\varepsilon^{\frac{1}{8(d_1+d_2)-17}} \textrm{ and } \triangle_2=\varepsilon^{\frac{4(d_1+d_2)-10}{8(d_1+d_2)-17}}. \label{s5-12}
\end{equation}
It is obvious that
\begin{equation}
|\Omega_2|\lesssim \triangle_1. \label{s5-9}
\end{equation}
By using Carbery, Christ and Wright's multidimensional van der Corput theorem \cite[Theorem 7.1]{CCW99}, we find that
\begin{equation}
|\Omega_3|\lesssim \triangle_2^{\frac{1}{\mathfrak{d}}} \leq \triangle_2^{\frac{1}{4(d_1+d_2)-10}}. \label{s5-10}
\end{equation}

It remains to estimate $|\Omega(\varepsilon)\cap\Omega_1|$.
We consider a mapping $\mathbf{T}: [1/4,4]^3\rightarrow\mathbb{R}^3$, $\mathbf{t}=(t_1,t_2,t_3)\mapsto(u,v,w)$ defined by
\begin{equation}\label{s5-14}
	\begin{cases}
		u=-\widetilde{P}_1(t_1)+\widetilde{P}_1(t_2)-\widetilde{P}_1(t_3),\\
		v=\widetilde{P}_2(t_1)-\widetilde{P}_2(t_2)+\widetilde{P}_2(t_3),\\
		w=t_1.
	\end{cases}
\end{equation}
Straightforward computation shows that if $|l|>\Gamma$ is sufficiently large then
\begin{equation}
	\left|\frac{\partial(u,v,w)}{\partial(t_1,t_2,t_3)}\right|(\mathbf{t})
=\left|\widetilde{P}_1'(t_2)\widetilde{P}_2'(t_3)-\widetilde{P}_1'(t_3)\widetilde{P}_2'(t_2)\right|
\asymp 2^{-\mathfrak{c}|l|}|t_2-t_3|, \label{s5-16}
\end{equation}
where $\mathfrak{c}=0$ if $l<-\Gamma$ and $d_1\neq d_2$ or if $l>\Gamma$ and $\sigma_1\neq \sigma_2$; $\mathfrak{c}$ is a positive integer (depending only on $P_1$, $P_2$ and the sign of $l$) if $l<-\Gamma$ and $d_1=d_2$ or if $l>\Gamma$ and $\sigma_1=\sigma_2$.

We choose balls that intersect $\Omega_1$ from a family of balls centered at rescaled lattice points $(\mathfrak{c}_1 2^{-2\mathfrak{c}l}\triangle_1\triangle_2)\mathbb{N}^3$ with radius\footnote{Roughly speaking, we would like to work with small balls below because $T$ is locally bijective and we want to reduce the estimation of $|\Omega({\varepsilon}) \cap \mathcal{B}|$ to that of a sublevel set which is contained in an \textit{interval} $I_{\mathcal{B}}^{u,v}$ (defined below).} $\mathfrak{c}_1 2^{-2\mathfrak{c}l}\triangle_1\triangle_2$ (with a constant $\mathfrak{c}_1$ to be determined below). Hence $\Omega_1$ is covered by
\begin{equation}
O((2^{-2\mathfrak{c}l}\triangle_1\triangle_2)^{-3}) \label{s5-20}
\end{equation}
such balls. Let $\mathcal{B}=B(\mathbf{t}_0, \mathfrak{c}_1 2^{-2\mathfrak{c}l}\triangle_1\triangle_2)$ be any one of them and $\mathcal{B}^*=B(\mathbf{t}_0, \mathfrak{c}_2 2^{-\mathfrak{c}l}\triangle_2)$ a concentric ball with larger radius where $\mathfrak{c}_2$ is chosen so small that
\begin{enumerate}[\upshape (i)]
\item  for any $\mathbf{t}\in \mathcal{B}^*$ we have $|\widetilde{Q}(\mathbf{t})|>\triangle_2/2$ and $|t_2-t_3|\geq \triangle_1/2$ by the mean value theorem;
\item  the mapping $\mathbf{T}$ is a bijection from $\mathcal{B}^*$ to $\mathbb{R}^3$ such that
    \begin{equation*}
B\left(\mathbf{T}(\mathbf{t}_0), A2^{-2\mathfrak{c}l}\triangle_1\triangle_2\right)\subset \mathbf{T}\left(\mathcal{B}^*\right)
    \end{equation*}
    for some constant $A$ by Lemma \ref{app-2}.
\end{enumerate}
Since $\mathbf{T}(\mathcal{B})$ is contained in a ball with radius $\asymp \mathfrak{c}_1 2^{-2\mathfrak{c}l}\triangle_1\triangle_2$ by Lemma \ref{app-2}, we can choose $\mathfrak{c}_1$ so small that
\begin{equation}
\mathbf{T}(\mathcal{B})\subset B\left(\mathbf{T}(\mathbf{t}_0), A2^{-2\mathfrak{c}l}\triangle_1\triangle_2\right).\label{s5-17}
\end{equation}

The estimate of $|\Omega({\varepsilon})\cap\Omega_1|$ is then reduced to that of $|\Omega({\varepsilon}) \cap \mathcal{B}|$. By \eqref{s5-16} and \eqref{s5-17}, we have
\begin{align}
	|\Omega({\varepsilon}) \cap \mathcal{B}|&=\int_{\mathcal{B}}\!\mathbf{1}_{\Omega(\varepsilon)}(\mathbf{t})\,\mathrm{d}\mathbf{t}\nonumber \\
&=\int_{\mathbf{T}(\mathcal{B})}\!\mathbf{1}_{\mathbf{T}(\Omega(\varepsilon))}(u,v,w)\left|\frac{\partial(u,v,w)}{\partial(t_1,t_2,t_3)}\right|^{-1}
\,\mathrm{d}u\mathrm{d}v\mathrm{d}w\nonumber \\
	&\lesssim 2^{\mathfrak{c}l}\triangle_1^{-1}\left(2^{-2\mathfrak{c}l}\triangle_1\triangle_2\right)^2\!\!\sup_{(u,v)\in\Sigma_{\mathcal{B}}}\!\!\left|\left\{w\in I_{\mathcal{B}}^{u,v}:|F(\mathbf{t}(u,v,w))|\leq\varepsilon\right\}\right|, \label{s5-21}
\end{align}
where
\begin{equation*}
\Sigma_{\mathcal{B}}:=\left\{(u,v)\in\mathbb{R}^2:(u,v,w)\in\mathbf{T}(\mathcal{B})\text{ for some }w\in\mathbb{R}\right\}
\end{equation*}
is the projection of $\mathbf{T}(\mathcal{B})$ onto the $(u,v)$-plane, and
\begin{equation*}
I_{\mathcal{B}}^{u,v}:=\left\{w\in\mathbb{R}:(u,v,w)\in B\left(\mathbf{T}(\mathbf{t}_0), A2^{-2\mathfrak{c}l}\triangle_1\triangle_2\right)\right\},
\end{equation*}
a slice of the ball $B\left(\mathbf{T}(\mathbf{t}_0), A2^{-2\mathfrak{c}l}\triangle_1\triangle_2\right)$, is a bounded interval if $(u,v)\in \Sigma_{\mathcal{B}}$.

We will next show that for any fixed $(u_0,v_0)\in\Sigma_{\mathcal{B}}$ we have
\begin{equation}
	\left|\left\{w\in I_{\mathcal{B}}^{u_0,v_0} : |F(\mathbf{t}(u_0,v_0,w))|\leq\varepsilon\right\}\right|
\lesssim 2^{(\mathfrak{d}-4(\sigma_1+\sigma_2)+10-\mathfrak{c})l}\varepsilon\triangle_2^{-1}. \label{s5-18}
\end{equation}
To achieve this we first study the function
\begin{align*}
	G(w):&=F(\mathbf{t}(u_0,v_0,w))\\
&=\alpha(\bar{\mathbf{z}})H(t_1)^{-1}H(t_2)H(t_3)^{-1}-\beta(\bar{\mathbf{z}}+(u_0,v_0)),
\end{align*}
which is smooth in $w\in I_{\mathcal{B}}^{u_0,v_0}$. Notice that, by implicit differentiation, we get from \eqref{s5-14} that
\begin{equation*}
	\begin{cases}
		\frac{\partial t_1}{\partial w}=1,\\
		\frac{\partial t_2}{\partial w}=\frac{\widetilde{P}_1'(t_1)\widetilde{P}_2'(t_3)-\widetilde{P}_1'(t_3)\widetilde{P}_2'(t_1)}{\widetilde{P}_1'(t_2)\widetilde{P}_2'(t_3)-\widetilde{P}_1'(t_3)\widetilde{P}_2'(t_2)},\\
		\frac{\partial t_3}{\partial w}=\frac{\widetilde{P}_1'(t_1)\widetilde{P}_2'(t_2)-\widetilde{P}_1'(t_2)\widetilde{P}_2'(t_1)}{\widetilde{P}_1'(t_2)\widetilde{P}_2'(t_3)-\widetilde{P}_1'(t_3)\widetilde{P}_2'(t_2)}.
	\end{cases}
\end{equation*}
By using these derivatives, the chain rule, \eqref{s5-16} and the definitions of $H$, $Q$ and $\widetilde{Q}$, we find that
\begin{align*}
\left|G'(w)\right|&=\left|\frac{\alpha(\bar{\mathbf{z}})}{(\widetilde{P}_2'(t_1)\widetilde{P}_1'(t_2)\widetilde{P}_2'(t_3))^2}\cdot
\frac{Q(t_1,t_2,t_3)+Q(t_2,t_3,t_1)+Q(t_3,t_1,t_2)}
{\widetilde{P}_1'(t_2)\widetilde{P}_2'(t_3)-\widetilde{P}_1'(t_3)\widetilde{P}_2'(t_2)} \right|\\
&\asymp 2^{(\mathfrak{c}-(\mathfrak{d}-4(\sigma_1+\sigma_2)+10))l}\left|\widetilde{Q}(\mathbf{t})\right|\\
&\gtrsim   2^{(\mathfrak{c}-(\mathfrak{d}-4(\sigma_1+\sigma_2)+10))l}\triangle_2.
\end{align*}
Then by the mean value theorem we obtain \eqref{s5-18} immediately.

To conclude, by \eqref{s5-20}, \eqref{s5-21} and \eqref{s5-18} we obtain
\begin{equation}
	|\Omega({\varepsilon})\cap\Omega_1|\leq \sum_{\mathcal{B}}|\Omega({\varepsilon}) \cap \mathcal{B}|\lesssim 2^{(\mathfrak{d}-4(\sigma_1+\sigma_2)+10+2\mathfrak{c})l}\frac{\varepsilon}{\triangle_1^2 \triangle_2^2}. \label{s5-19}
\end{equation}
Therefore by \eqref{s5-9}, \eqref{s5-10}, \eqref{s5-19} and \eqref{s5-12} we get
\begin{align*}
	|\Omega(\varepsilon)|&\lesssim 2^{(\mathfrak{d}-4(\sigma_1+\sigma_2)+10+2\mathfrak{c})l}\frac{\varepsilon}{\triangle_1^2 \triangle_2^2}+\triangle_1+\triangle_2^{\frac{1}{4(d_1+d_2)-10}}\\
                         &\lesssim 2^{(\mathfrak{d}-4(\sigma_1+\sigma_2)+10+2\mathfrak{c})l}\varepsilon^{\frac{1}{8(d_1+d_2)-17}},
\end{align*}
which in turn gives
\begin{equation*}
	|\mathcal{E}|\lesssim |\mathcal{A}|^{1/7}\lesssim 2^{\frac{1}{7}(\mathfrak{d}-4(\sigma_1+\sigma_2)+10+2\mathfrak{c})l}\varepsilon^{\frac{1}{7(8(d_1+d_2)-17)}},
\end{equation*}
as desired.
\end{proof}

\begin{remark} \label{remark3.2}
In fact, the expression
\begin{equation*}
Q(t_1,t_2,t_3)+Q(t_2,t_3,t_1)+Q(t_3,t_1,t_2)
\end{equation*}
shown up in the definition of $\widetilde{Q}$ contains a nice structure that one can exploit and take advantage of. Its dominating terms contain a factor of homogeneous polynomial which can be factorized by Lemma \ref{app-1}.

For example, if $l>\Gamma$ is sufficiently large (so that $\widetilde{P}_i(t)$ behaves like $t^{\sigma_i}$) and $\sigma_1\neq\sigma_2$, we are able to take advantage of this special structure (see Lemma \ref{s5-5}) and improve the exponent of the $\varepsilon$ term in \eqref{s5-11} to $1/35$ (see the following proposition). If $\sigma_1=\sigma_2$, we believe the same improvement should still be true but we are unable to prove it at this moment.
\end{remark}

\begin{proposition}\label{l4-1}
Let $P_1$ and $P_2$ be linearly independent polynomials  with zero constant term, $I=[1/2,2]$, $K\subset\mathbb{R}^2\times I$ a compact set and $\alpha, \beta$: $\mathbb{R}^2\rightarrow\mathbb{R}$  measurable functions with either $|\alpha|\asymp 1$ or $|\beta|\asymp1$. We further assume $\sigma_1\neq \sigma_2$ if $l>\Gamma$. There exists a nonnegative constant $c$ (which is a positive integer depending only on $P_1$ and $P_2$ if $l<-\Gamma$ and $d_1=d_2$, is $0$ otherwise) such that if $|l|>\Gamma$ is sufficiently large then
\begin{equation}
\begin{split}\label{s5-1}
		&\left|\left\lbrace(x,y,t)\in \! K:\left|\widetilde{P}_1'(t)\alpha(x+\widetilde{P}_1(t),y)-\widetilde{P}_2'(t)\beta(x,y+\widetilde{P}_2(t))\right|\leq \! \varepsilon\!\right\rbrace\right|\\
\lesssim & 2^{c|l|}\varepsilon^{1/35}
\end{split}
\end{equation}
for all $\varepsilon\in (0,1]$. The implicit constant depends only on $K$, $P_1$ and $P_2$, but not on measurable functions $\alpha$ and $\beta$.
\end{proposition}

\begin{proof}
As before the estimation is reduced to that of $\Omega({\varepsilon})$ (defined by \eqref{s5-13}). We observe that
\begin{equation*}
	\Omega({\varepsilon})\subseteq\bigcup_{i=1}^5\Omega_i,
\end{equation*}
where we denote, with $a=1/10$ and $b=1/5$, that
\begin{align*}
	\Omega_1&=\{\mathbf{t}\in I^3:|t_1-t_2|\geq \varepsilon^a/2,|t_3-t_1|\geq \varepsilon^a/2,|t_2-t_3|\geq \varepsilon^b\},\\
	\Omega_2&=\{\mathbf{t}\in I^3:|t_1-t_2|\geq \varepsilon^a/2,|t_2-t_3|\geq \varepsilon^a/2,|t_3-t_1|\geq \varepsilon^b\},\\
	\Omega_3&=\{\mathbf{t}\in I^3:|t_2-t_3|\geq \varepsilon^a/2,|t_3-t_1|\geq \varepsilon^a/2,|t_1-t_2|\geq \varepsilon^b\},\\
	\Omega_4&=\{\mathbf{t}\in I^3:|t_1-t_2|\leq \varepsilon^a,|t_2-t_3|\leq \varepsilon^a,|t_3-t_1|\leq \varepsilon^a\},\\
	\Omega_5&=\{\mathbf{t}\in I^3:|t_1-t_2|\leq \varepsilon^b\text{ or }|t_2-t_3|\leq \varepsilon^b\text{ or }|t_3-t_1|\leq \varepsilon^b\}.
\end{align*}
It is geometrically evident that $|\Omega_4|\lesssim \varepsilon^{2a}$ and $|\Omega_5|\lesssim \varepsilon^b$.

It remains to estimate $|\Omega({\varepsilon})\cap\Omega_i|$ with $1\leq i\leq 3$. We may set $i=1$ while the other two cases can be handled similarly (with the third equation in \eqref{s5-14} replaced by $w=t_i$). As before we consider the mapping $\mathbf{T}$ given by \eqref{s5-14} with Jacobian given by \eqref{s5-16}.

We choose balls that intersect $\Omega_1$ from a family of balls centered at $(\mathfrak{c}_1 2^{-2\mathfrak{c}|l|}\varepsilon^{2b})\mathbb{N}^3$ with radius  $\mathfrak{c}_1 2^{-2\mathfrak{c}|l|}\varepsilon^{2b}$ (with $\mathfrak{c}$ being the constant appearing in \eqref{s5-16} and $\mathfrak{c}_1$ a constant to be determined below). Hence $\Omega_1$ is covered by
\begin{equation*}
O\left(2^{6\mathfrak{c}|l|}\varepsilon^{-6b}\right)
\end{equation*}
such balls. Let $\mathcal{B}=B(\mathbf{t}_0, \mathfrak{c}_1 2^{-2\mathfrak{c}|l|}\varepsilon^{2b})$ be any one of them. Denote by $\mathcal{B}^*=B(\mathbf{t}_0, \mathfrak{c}_2 2^{-\mathfrak{c}|l|}\varepsilon^{b})$ a concentric ball with larger radius where $\mathfrak{c}_2$ is chosen so small that
\begin{enumerate}[\upshape (i)]
\item  for any $\mathbf{t}\in \mathcal{B}^*$ we have $|t_1-t_2|\geq \varepsilon^a/4$, $|t_3-t_1|\geq \varepsilon^a/4$ and $|t_2-t_3|\geq \varepsilon^b/2$;
\item  the mapping $\mathbf{T}$ is a bijection from $\mathcal{B}^*$ to $\mathbb{R}^3$ such that
    \begin{equation*}
B\left(\mathbf{T}(\mathbf{t}_0), A2^{-2\mathfrak{c}|l|}\varepsilon^{2b}\right)\subset \mathbf{T}\left(\mathcal{B}^*\right)
    \end{equation*}
    for some constant $A$ (by Lemma \ref{app-2}).
\end{enumerate}
Since $\mathbf{T}(\mathcal{B})$ is contained in a ball with radius $\asymp \mathfrak{c}_1 2^{-2\mathfrak{c}|l|} \varepsilon^{2b}$ by Lemma \ref{app-2}, we can choose $\mathfrak{c}_1$ so small that
\begin{equation}
\mathbf{T}(\mathcal{B})\subset B\left(\mathbf{T}(\mathbf{t}_0), A2^{-2\mathfrak{c}|l|}\varepsilon^{2b}\right).\label{s5-8}
\end{equation}

It suffices to estimate $|\Omega({\varepsilon}) \cap \mathcal{B}|$. By \eqref{s5-16} and \eqref{s5-8}, we have
\begin{equation*}
	|\Omega({\varepsilon}) \cap \mathcal{B}|\lesssim 2^{\mathfrak{c}|l|}\varepsilon^{-b}\left(2^{-2\mathfrak{c}|l|}\varepsilon^{2b}\right)^2\sup_{(u,v)\in\Sigma_{\mathcal{B}}}\left|\left\{w\in I_{\mathcal{B}}^{u,v}:|F(\mathbf{t}(u,v,w))|\leq\varepsilon\right\}\right|,
\end{equation*}
where
\begin{equation*}
\Sigma_{\mathcal{B}}:=\left\{(u,v)\in\mathbb{R}^2:(u,v,w)\in\mathbf{T}(\mathcal{B})\text{ for some }w\in\mathbb{R}\right\}
\end{equation*}
is the projection of $\mathbf{T}(\mathcal{B})$ onto the $(u,v)$-plane, and
\begin{equation*}
I_{\mathcal{B}}^{u,v}:=\left\{w\in\mathbb{R}:(u,v,w)\in B\left(\mathbf{T}(\mathbf{t}_0), A2^{-2\mathfrak{c}|l|}\varepsilon^{2b}\right)\right\},
\end{equation*}
a slice of the ball $B(\mathbf{T}(\mathbf{t}_0), A2^{-2\mathfrak{c}|l|}\varepsilon^{2b})$, is a bounded interval if $(u,v)\in \Sigma_{\mathcal{B}}$.

We claim that for any fixed $(u_0,v_0)\in\Sigma_{\mathcal{B}}$ we have
\begin{equation}
	\left|\left\{w\in I_{\mathcal{B}}^{u_0,v_0} : |F(\mathbf{t}(u_0,v_0,w))|\leq\varepsilon\right\}\right|
\lesssim 2^{\mathfrak{c}|l|}\varepsilon^{1-2a}. \label{s5-7}
\end{equation}
Indeed, let $G(w)=F(\mathbf{t}(u_0,v_0,w))$. By implicit differentiation and \eqref{s5-16}, we find that
\begin{equation*}
\left|G'(w)\right|\asymp \frac{|Q(t_1,t_2,t_3)+Q(t_2,t_3,t_1)+Q(t_3,t_1,t_2)|}
{2^{-\mathfrak{c}|l|}|t_2-t_3|},
\end{equation*}
with $Q(\mathbf{t})$ given by \eqref{s5-22}. For the numerator we have
\begin{lemma} \label{s5-5}
If $|l|>\Gamma$ is sufficiently large, then
\begin{equation*}
|Q(t_1,t_2,t_3)+Q(t_2,t_3,t_1)+Q(t_3,t_1,t_2)|\asymp 2^{-2\mathfrak{c}|l|}|(t_1-t_2)(t_1-t_3)(t_2-t_3)|,
\end{equation*}
where $\mathfrak{c}$ is the constant appearing in \eqref{s5-16} which is in $\mathbb{N}$ if $l<-\Gamma$ and $d_1=d_2$, is $0$ if $l<-\Gamma$ and $d_1\neq d_2$ or if $l>\Gamma$ and $\sigma_1\neq \sigma_2$.
\end{lemma}
\noindent We will prove this lemma later. It follows from this lemma and $\mathbf{t}\in \mathcal{B}^*$ that
\begin{equation*}
	|G'(w)|\asymp 2^{-\mathfrak{c}|l|}|t_1-t_2||t_1-t_3|\gtrsim 2^{-\mathfrak{c}|l|}\varepsilon^{2a}.
\end{equation*}
Then by the mean value theorem we obtain \eqref{s5-7} immediately.

To conclude we obtain
\begin{equation*}
	|\Omega({\varepsilon})\cap\Omega_1|\leq \sum_{\mathcal{B}}|\Omega({\varepsilon}) \cap \mathcal{B}|\lesssim 2^{4\mathfrak{c}|l|}\varepsilon^{1-2a-3b}
\end{equation*}
and similarly $|\Omega({\varepsilon})\cap\Omega_i|\lesssim 2^{4\mathfrak{c}|l|}\varepsilon^{1-2a-3b}$ for $i=2,3$. Therefore
\begin{equation*}
	|\Omega(\varepsilon)|\lesssim 2^{4\mathfrak{c}|l|}\left(\varepsilon^{1-2a-3b}+\varepsilon^{2a}+\varepsilon^b\right)\lesssim 2^{4\mathfrak{c}|l|}\varepsilon^{1/5},
\end{equation*}
which in turn gives
\begin{equation*}
	|\mathcal{E}|\lesssim |\mathcal{A}|^{1/7}\lesssim 2^{4\mathfrak{c}|l|/7}\varepsilon^{1/35},
\end{equation*}
as desired.
\end{proof}

\begin{proof}[Proof of Lemma \ref{s5-5}]
We first consider the case $l<-\Gamma$ and $d_1\neq d_2$ (while the case $l>\Gamma$ and $\sigma_1\neq \sigma_2$ is similar). When $\Gamma$ is large, the polynomial $\widetilde{P}_j(t)$ behaves like $t^{d_j}$. The main term of $Q(t_1,t_2,t_3)+Q(t_2,t_3,t_1)+Q(t_3,t_1,t_2)$ is equal to
\begin{align*}
&(d_2-d_1)(d_1d_2a_{d_1}b_{d_2})^4\left(t_1t_2t_3\right)^{d_1+d_2-3}\cdot \\
&\quad \left(t_2^{d_1}t_3^{d_2}+t_3^{d_1}t_1^{d_2}
+t_1^{d_1}t_2^{d_2}-t_3^{d_1}t_2^{d_2}-t_1^{d_1}t_3^{d_2}-t_2^{d_1}t_1^{d_2}\right)\\
\asymp &|(t_1-t_2)(t_1-t_3)(t_2-t_3)|
\end{align*}
by Lemma \ref{app-1}, while the rest terms are of size
\begin{equation*}
O\left(2^{-|l|}|(t_1-t_2)(t_1-t_3)(t_2-t_3)|\right).
\end{equation*}
Hence the desired estimate holds with $\mathfrak{c}=0$.

We next consider the case $l<-\Gamma$ and $d:=d_1=d_2$. Since $P_1$ and $P_2$ are linearly independent, we can write
\begin{equation*}
P_2(t)=\frac{b_d}{a_d}P_1(t)+c_{\varrho}t^{\varrho}+\mathcal{E}(t)
\end{equation*}
for some integer $\min\{\sigma_1, \sigma_2\}\leq \varrho<d$ and nonzero constant $c_{\varrho}$, where $\mathcal{E}(t)$ is a (possibly trivial) polynomial of $t$ with a degree $\leq \varrho-1$. We plug this formula of $P_2$ into the expression of $Q(t_1, t_2, t_3)$ to find some cancellation. The main term of $Q(t_1,t_2,t_3)+Q(t_2,t_3,t_1)+Q(t_3,t_1,t_2)$ is equal to
\begin{align*}
&2^{-2(d-\varrho)|l|}(\varrho-d)d^6\varrho^2 a_d^4 b_d^2 c_{\varrho}^2  \left(t_1 t_2 t_3\right)^{d+\varrho-3}\cdot \\ &\quad \left(t_2^{2d-\varrho}t_3^d+t_3^{2d-\varrho}t_1^d+t_1^{2d-\varrho}t_2^d
-t_3^{2d-\varrho}t_2^d-t_1^{2d-\varrho}t_3^d-t_2^{2d-\varrho}t_1^d\right)\\
\asymp & 2^{-2(d-\varrho)|l|}|(t_1-t_2)(t_1-t_3)(t_2-t_3)|,
\end{align*}
while the rest terms are of size
\begin{equation*}
O\left(2^{-|l|-2(d-\varrho)|l|}|(t_1-t_2)(t_1-t_3)(t_2-t_3)|\right).
\end{equation*}
Thus the desired estimate holds with $\mathfrak{c}=d-\varrho$.
\end{proof}


\section{Proof of Theorem \ref{thm1}}\label{sec4}
We first introduce the Littlewood-Paley decomposition. Let $\chi\in C_c^\infty(\mathbb{R})$ be a nonnegative radially decreasing function supported on $[-2,2]$ which equals $1$ on $[-1,1]$. Denote $\chi_l(t)=\chi(2^{-l}t)$, $\psi(t)=\chi(t)-\chi(2t)$ and $\psi_l(t)=\psi(2^{-l}t)$. Then we have
\begin{equation*}
	\sum_{l\in\mathbb{Z}}\psi_l(t)=1 \textrm{ for all $t\neq0$}.
\end{equation*}
We use partial Littlewood-Paley operators $\Delta_l^{(j)}$ (associated with $\psi_l$) and partial sums $S_l^{(j)}$ given by
\begin{equation*}
	\Delta_l^{(j)}f=\left(\psi_l\right)^\vee*_j f \textrm{ and }  S_l^{(j)}f=\left(\chi_l\right)^\vee*_j f
\end{equation*}
for $j=1,2$ and $l\in\mathbb{Z}$.

We then have the decomposition
\begin{equation*}
T=\sum_{l\in\mathbb{Z}}T_l
\end{equation*}
with
\begin{equation}\label{s2-41}
	T_l(f_1,f_2)(x,y)=\int_{\mathbb{R}}\!f_1\left(x+P_1(t),y\right)f_2\left(x,y+P_2(t)\right)\psi\left(2^l t\right)t^{-1}\,\mathrm{d}t.
\end{equation}
By applying the Littlewood-Paley decomposition to $f_1$ and $f_2$, we rewrite $T$ as
\begin{equation*}
	T=T^L+T^M+T^H+\sum_{|l|\leq \Gamma}T_l
\end{equation*}
for some large $\Gamma$, where for $\omega\in\{L,M,H\}$,
\begin{equation*}
	T^\omega=\sum_{|l|>\Gamma}T_l^\omega
\end{equation*}
with
\begin{equation*}
	\begin{split}
	T_l^\omega(f_1,f_2)&=\sum_{k=(k_1,k_2)\in\mathfrak{F}_\omega}T_l\left(\Delta_{\mathfrak{r}_1 l+k_1}^{(1)}f_1,\Delta_{\mathfrak{r}_2 l+k_2}^{(2)}f_2\right),	\\
	\mathfrak{F}_L&=\left\{k\in\mathbb{Z}^2:\max(k_1,k_2)\leq\mathcal{K}\right\},\\
	\mathfrak{F}_M&=\left\{k\in\mathbb{Z}^2:\max(k_1,k_2)>\mathcal{K},|k_1-k_2|\geq\mathcal{K}\right\},\\
	\mathfrak{F}_H&=\left\{k\in\mathbb{Z}^2:\max(k_1,k_2)>\mathcal{K},|k_1-k_2|<\mathcal{K}\right\},
\end{split}
\end{equation*}
for some large $\mathcal{K}$. Since it is obvious that
\begin{equation*}
	\|T_l(f_1,f_2)\|_r\leq C\|f_1\|_p\|f_2\|_q
\end{equation*}
for all $l\in\mathbb{Z}$, we only need to estimate each $T^\omega$.

For the associated bilinear maximal operator $M$,  define
\begin{equation*}
M_l(f_1,f_2)(x,y)=\int_\mathbb{R}\!f_1(x+P_1(t),y)f_2(x,y+P_2(t))\psi(2^lt)2^l\,\mathrm{d}t.
\end{equation*}
It suffices to estimate  $\sup_{l\in\mathbb{Z}}|M_l|$ instead. Notice that
 \begin{equation*}
 	\left|M_l(f_1,f_2)(x,y)\right|\leq 2^\Gamma\int_{\mathbb{R}}\!\left|f_1\left(x+P_1(t),y\right)f_2\left(x,y+P_2(t)\right)\chi\left(2^{-\Gamma-1}t\right)\right|\,\mathrm{d}t
 \end{equation*}
 for all $|l|\leq\Gamma$, hence
 \begin{equation*}
 	\left\|\sup_{|l|\leq \Gamma}\left|M_l(f_1,f_2)\right|\right\|_r\lesssim_\Gamma \|f_1\|_p\|f_2\|_q.
 \end{equation*}
For any fixed $|l|>\Gamma$, decompose
\begin{equation*}
	M_l=M_l^{L}+M_l^{M}+M_l^{H},
\end{equation*}
where
\begin{equation*}
	M_l^{\omega}(f_1,f_2)=\sum_{k\in\mathfrak{F}_\omega}M_l\left(\Delta_{\mathfrak{r}_1 l+k_1}^{(1)}f_1,\Delta_{\mathfrak{r}_2 l+k_2}^{(2)}f_2\right).
\end{equation*}
It then suffices to estimate each $\sup_{|l|>\Gamma}|M_l^{\omega}|$.

\subsection{Estimation of $T^H$ and $\sup_{|l|>\Gamma}|M_l^{H}|$.}\label{sec4-1}
For each $k=(k_1,k_2)\in\mathbb{Z}^2$, denote
\begin{align*}
	T^{(k)}(f_1,f_2)=\sum_{|l|>\Gamma}T_l\left(\Delta_{\mathfrak{r}_1 l+k_1}^{(1)}f_1,\Delta_{\mathfrak{r}_2 l+k_2}^{(2)}f_2\right).
\end{align*}
We will prove for all $k\in\mathfrak{F}_H$, $p,q\in(1,\infty)$ and $r\in [1,\infty)$ with $p^{-1}+q^{-1}=r^{-1}$,  that
\begin{equation}\label{s2-01}
	\left\|T^{(k)}(f_1,f_2)\right\|_r\lesssim 2^{-c_{p,q}|k|}\|f_1\|_p\|f_2\|_q
\end{equation}
for some constant $c_{p,q}>0$. Then we can get the bound
\begin{equation}\label{s2-05}
	\left\|T^H(f_1,f_2)\right\|_r\lesssim\|f_1\|_p\|f_2\|_q
\end{equation}
for all  $p,q\in(1,\infty)$ and $r\in [1,\infty)$ with $p^{-1}+q^{-1}=r^{-1}$ by summing over $k\in\mathfrak{F}_H$.

We remark that the definition of $M_l$ is essentially the same as that of $T_l$ if we do not use the mean zero property of the function $\psi(2^lt)t^{-1}$ in \eqref{s2-41}. By a similar argument one can prove the bound \eqref{s2-01} with $T^{(k)}$ replaced by $M^{(k)}$ and then obtain the bound \eqref{s2-05} for the operator $\sup_{|l|>\Gamma}|M_l^{H}|$.  Therefore we will only consider $T^{(k)}$ below.

In fact, for $k\in\mathfrak{F}_H$ we will prove that
\begin{equation}\label{s2-02}
	\left\|T^{(k)}(f_1,f_2)\right\|_1\lesssim 2^{-c|k|}\|f_1\|_2\|f_2\|_2
\end{equation}
for some constant $c>0$ and
\begin{equation}\label{s2-03}
	\left\|T^{(k)}(f_1,f_2)\right\|_r\lesssim |k|^{4}\|f_1\|_p\|f_2\|_q
\end{equation}
for all $p,q\in(1,\infty)$ and $r\in [1,\infty)$ with $p^{-1}+q^{-1}=r^{-1}$. Then an interpolation argument yields \eqref{s2-01}.

To prove the bound \eqref{s2-03}, we only need to prove the following generalized version of \cite[Lemma 2.1]{CDR21} (which provided for each single scale piece $T_0$ a bound in terms of maximal functions). Since general polynomials are lack of homogeneity, we have to provide bounds for all pieces $T_l$ with sufficiently large $|l|$. Then arguing as in \cite[P. 11]{CDR21} easily yields \eqref{s2-03}.

Recall that the shifted (dyadic) maximal function is given by
\begin{equation*}
	M_qf(x)=\sup_{l\in\mathbb{Z}}2^{-l}\int_{q2^l}^{(q+1)2^l}|f(x+t)|\,\mathrm{d}t
\end{equation*}
for $q,x\in\mathbb{R}$, and $M_q^{(j)}$ represents the operator $M_q$ applied in the $j$-th coordinate. Then we have
\begin{lemma} \label{lemma4.1}
For all $|l|>\Gamma$ and $k\in\mathfrak{F}_H$,
	\begin{equation}
		\left|T_l\left(\Delta_{\mathfrak{r}_1 l+k_1}^{(1)}f_1,\Delta_{\mathfrak{r}_2 l+k_2}^{(2)}f_2\right)\right|\lesssim\sum_{i\in\mathcal{I}}c_{i}\left(M_{h_{1,i}}^{(1)}f_1\right)\left(M_{h_{2,i}}^{(2)}f_2\right),\label{s5-24}
	\end{equation}
where $\mathcal{I}$ is a countable index set and $c_i>0$, $h_{j,i}\in\mathbb{R}$ with
\begin{equation}
	\sum_{i\in\mathcal{I}}c_{i}\log^a\left(2+|h_{1,i}|\right)\log^b\left(2+|h_{2,i}|\right)\lesssim|k|^{a+b}\label{s5-25}
\end{equation}
for every $a,b>0$.
\end{lemma}

\begin{proof}
We may only consider the case $l>\Gamma$ (therefore, $\mathfrak{r}_j=\sigma_j$ for $j=1,2$) and $k_1\leq k_2$  while other cases can be treated similarly. By the definition \eqref{s2-41}, one can express $T_l\left(\Delta_{\sigma_1 l+k_1}^{(1)}f_1,\Delta_{\sigma_2 l+k_2}^{(2)}f_2\right)(x,y)$ as
\begin{align*}
	\int_{\mathbb{R}^3}\! &\psi_{\sigma_1 l+k_1}^{\vee}(u)f_1(x+P_1(2^{-l}t)-u,y)\cdot\\
	&\psi_{\sigma_2 l+k_2}^{\vee}(v)f_2(x,y+P_2(2^{-l}t)-v)\psi(t)t^{-1}\,\mathrm{d}t\mathrm{d}u\mathrm{d}v.
\end{align*}
In view of the support of $\psi$, we split the integration in $t$ over dyadic intervals
\begin{equation*}
	I_{q,k_1}=[q2^{-k_1},(q+1)2^{-k_1}] \textrm{ for } q\in \mathbb{Z}, 2^{k_1-1}\leq|q|\leq2^{k_1+1}.
\end{equation*}
Using the rapid decay of $\psi^{\vee}$ yields that $T_l\left(\Delta_{\sigma_1 l+k_1}^{(1)}f_1,\Delta_{\sigma_2 l+k_2}^{(2)}f_2\right)(x,y)$ is majorized by
\begin{equation}\label{s2-04}
		\sum_{n_1,n_2\in\mathbb{Z}}\prod_{j=1,2}(1+|n_j|)^{-N}\sum_{2^{k_1-1}\leq|q|\leq2^{k_1+1}}\int_{I_{q,k_1}}\!H_1(t,x,y)H_2(t,x,y)\,\mathrm{d}t,
\end{equation}
where
\begin{equation}
	H_1(t,x,y)=2^{\sigma_1 l+k_1}\int_{n_1 2^{-(\sigma_1 l+k_1)}+P_1(2^{-l}t)}^{(n_1+1)2^{-(\sigma_1 l+k_1)}+P_1(2^{-l}t)}|f_1|\left(x+u,y\right)\,\mathrm{d}u \label{s5-26}
\end{equation}
and
\begin{equation}
	H_2(t,x,y)=2^{\sigma_2 l+k_2}\int_{n_2 2^{-(\sigma_2 l+k_2)}+P_2(2^{-l}t)}^{(n_2+1)2^{-(\sigma_2l+k_2)}+P_2(2^{-l}t)}|f_2|\left(x,y+v\right) \,\mathrm{d}v. \label{s5-27}
\end{equation}
Denote
\begin{equation*}
	A_{q,n_1}=n_1+2^{k_1}\min\left\{\widetilde{P}_1\left(q 2^{-k_1}\right), \widetilde{P}_1\left((q+1) 2^{-k_1}\right)\right\}
\end{equation*}
and
\begin{equation*}
	B_{q,n_2}=n_2+2^{k_2}\min\left\{\widetilde{P}_2\left(q 2^{-k_1}\right),\widetilde{P}_2\left((q+1) 2^{-k_1}\right)\right\}.
\end{equation*}
One can observe that if $t\in I_{q,k_1}$ then integration domains in \eqref {s5-26} and \eqref {s5-27} are contained in finitely many dyadic intervals
\begin{equation*}
\bigcup_{j_1=1}^{O(1)} 2^{-(\sigma_1 l+k_1)} [A_{q,n_1}+j_1-1, A_{q,n_1}+j_1]
\end{equation*}
and
\begin{equation*}
	\bigcup_{j_2=1}^{O(2^\mathcal{K})} 2^{-(\sigma_2 l+k_2)} [B_{q,n_2}+j_2-1, B_{q,n_2}+j_2]
\end{equation*}
respectively. Hence $H_1$ and $H_2$ are both bounded by sums of finitely many shifted dyadic maximal functions, namely
\begin{equation*}
	H_1(t,x,y)\lesssim \sum_{j_1=1}^{O(1)} M_{A_{q,n_1}+j_1-1}^{(1)}f_1(x,y)
\end{equation*}
and
\begin{equation*}
	H_2(t,x,y)\lesssim\sum_{j_2=1}^{O(2^\mathcal{K})} M_{B_{q,n_2}+j_2-1}^{(2)}f_1(x,y).
 \end{equation*}
Plugging these two bounds above in \eqref{s2-04} and rewriting the resulting bound into a countable summation (with appropriate $c_i$ and $h_{j,i}$) give the right side of \eqref{s5-24}.  The estimate \eqref{s5-25} then follows easily. This finishes the proof.
\end{proof}

It remains to prove the bound \eqref{s2-02}. It suffices to prove that there exists a constant $c>0$ such that
\begin{equation}
	\left\|T_l\left(\Delta_{\mathfrak{r}_1 l+k_1}^{(1)}f_1,\Delta_{\mathfrak{r}_2 l+k_2}^{(2)}f_2\right)\right\|_1\lesssim 2^{-c|k|}\|f_1\|_2\|f_2\|_2 \label{s2-4}
\end{equation}
for all $|l|>\Gamma$ and $k=(k_1,k_2)\in\mathfrak{F}_H$. Indeed, if \eqref{s2-4} holds, then
\begin{equation*}
	\left\|T^{(k)}(f_1,f_2)\right\|_1\lesssim 2^{-c|k|}\sum_{|l|>\Gamma}\left\|\widetilde{\Delta}_{\mathfrak{r}_1 l+k_1}^{(1)}f_1\right\|_2\left\|\widetilde{\Delta}_{\mathfrak{r}_2 l+k_2}^{(2)}f_2\right\|_2\lesssim 2^{-c|k|}\|f_1\|_2\|f_2\|_2,
\end{equation*}
where  $\widetilde{\Delta}_l^{(j)}$ is a partial Littlewood-Paley operator satisfying $\Delta_l^{(j)}\widetilde{\Delta}_l^{(j)}=\Delta_l^{(j)}$.

By rescaling, it suffices to prove
\begin{equation}
	\left\|T_l'\left(\Delta_{k_1}^{(1)}f_1,\Delta_{k_2}^{(2)}f_2\right)\right\|_1\lesssim 2^{-c|k|}\|f_1\|_2\|f_2\|_2, \label{s2-6}
\end{equation}
where
\begin{equation}
	T_l'(f_1,f_2)(x,y)=\int_\mathbb{R}\!f_1\left(x+\widetilde{P}_1(t),y\right)f_2\left(x,y+\widetilde{P}_2(t)\right)\psi(t)t^{-1} \,\mathrm{d}t.\label{s2-5}
\end{equation}

Let  $\eta$ be a smooth nonnegative function on $\mathbb{R}^2$  supported in a small neighborhood of $[-1/2,1/2]^2$ such that $\sum_{m\in\mathbb{Z}^2}\eta_m=1$ with $\eta_m(z)=\eta(z-m)$. Then
\begin{equation*}
	\left\|T_l'\left(\Delta_{k_1}^{(1)}f_1,\Delta_{k_2}^{(2)}f_2\right)\right\|_1\leq\sum_{m\in\mathbb{Z}^2}\int_{\mathbb{R}^2}\!\left|T_l'\left(\Delta_{k_1}^{(1)}f_1,\Delta_{k_2}^{(2)}f_2\right)(x,y)\eta_m(x,y)\right|\mathrm{d}x\mathrm{d}y.
\end{equation*}
Note that there exists a constant $C>0$ such that if $(x,y,t)$ is in the support of $\eta_m(x,y)\psi(t)$, then
\begin{equation*}
	\left(x+\widetilde{P}_1(t),y\right),\left(x,y+\widetilde{P}_2(t)\right)\in [-C,C]^2+m.
\end{equation*}
We correspondingly choose a smooth nonnegative bump function $\widetilde{\eta}_m$ such that  $\widetilde{\eta}_m\equiv1$ on $[-C,C]^2+m$, $\sum_{m\in\mathbb{Z}^2}\widetilde{\eta}_m\lesssim1$ and $\|\widetilde{\eta}_m\|_{C^1}\lesssim1$. Denote $\zeta_m(x,y,t)=\eta_m(x,y)\psi(t)t^{-1}$ and a local operator
\begin{equation*}
	T_{l,m}(f_1,f_2)(x,y)=\int_\mathbb{R}\!f_1\left(x+\widetilde{P}_1(t),y\right)f_2\left(x,y+\widetilde{P}_2(t)\right)\zeta_m(x,y,t)\,\mathrm{d}t.
\end{equation*}
Then
\begin{align*}
	\left\|T_l'\left(\Delta_{k_1}^{(1)}f_1,\Delta_{k_2}^{(2)}f_2\right)\right\|_1&\leq\sum_{m\in\mathbb{Z}^2}\int_{\mathbb{R}^2}\!\left|T_{l,m}\left(\widetilde{\eta}_m\Delta_{k_1}^{(1)}f_1,\widetilde{\eta}_m\Delta_{k_2}^{(2)}f_2\right)(x,y)\right|\,\mathrm{d}x\mathrm{d}y\\
	&\leq \mathrm{I+II+III},
\end{align*}
where
\begin{align*}
	\mathrm{I}&=\sum_{m\in\mathbb{Z}^2}\int_{\mathbb{R}^2}\!\left|T_{l,m}\left(\Delta_{k_1}^{(1)}\left(\widetilde{\eta}_mf_1\right),\Delta_{k_2}^{(2)}(\widetilde{\eta}_mf_2)\right)(x,y)\right|\,\mathrm{d}x\mathrm{d}y,\\
	\mathrm{II}&=\sum_{m\in\mathbb{Z}^2}\int_{\mathbb{R}^2}\!\left|T_{l,m}\left(\widetilde{\eta}_m\Delta_{k_1}^{(1)}f_1-\Delta_{k_1}^{(1)}(\widetilde{\eta}_mf_1),\widetilde{\eta}_m\Delta_{k_2}^{(2)}f_2\right)(x,y)\right|\,\mathrm{d}x\mathrm{d}y,\\
	\mathrm{III}&=\sum_{m\in\mathbb{Z}^2}\int_{\mathbb{R}^2}\!\left|T_{l,m}\left(\Delta_{k_1}^{(1)}(\widetilde{\eta}_mf_1),\widetilde{\eta}_m\Delta_{k_2}^{(2)}f_2-\Delta_{k_2}^{(2)}(\widetilde{\eta}_mf_2)\right)(x,y)\right|\,\mathrm{d}x\mathrm{d}y.
\end{align*}

We begin with estimating the main term $\mathrm{I}$. By Theorem \ref{thm3} we obtain
\begin{equation*}
	\left\|T_{l,m}\left(\Delta_{k_1}^{(1)}(\widetilde{\eta}_mf_1),\Delta_{k_2}^{(2)}(\widetilde{\eta}_mf_2)\right)\right\|_1\lesssim 2^{-\sigma|k|}\left\|\widetilde{\eta}_mf_1\right\|_2\left\|\widetilde{\eta}_mf_2\right\|_2
\end{equation*}
for some absolute constant $\sigma>0$. Thus
\begin{equation*}
	\mathrm{I}\lesssim 2^{-\sigma|k|}\sum_{m\in\mathbb{Z}^2}\left\|\widetilde{\eta}_mf_1\right\|_2\left\|\widetilde{\eta}_mf_2\right\|_2\lesssim 2^{-\sigma|k|}\|f_1\|_2\|f_2\|_2.
\end{equation*}

As to the error term $\mathrm{II}$, by the mean value theorem we have
\begin{align*}
	&\left|\left(\widetilde{\eta}_m\Delta_{k_1}^{(1)}f_1-\Delta_{k_1}^{(1)}(\widetilde{\eta}_mf_1)\right)(x,y)\right|\\
\leq&\int_\mathbb{R}\!\left|\widetilde{\eta}_m(x,y)-\widetilde{\eta}_m(u,y)\right|\left|\left(\psi_{k_1}\right)^\vee(x-u)\right|\left|f_1(u,y)\right|\,\mathrm{d}u\\
\lesssim & \left|u\cdot\left(\psi_{k_1}\right)^\vee(u)\right|\!*_1\!|f_1|(x,y).
\end{align*}
We also have
\begin{equation*}
\left|\widetilde{\eta}_m\Delta_{k_2}^{(2)}f_2\right|\lesssim\left|\left(\psi_{k_2}\right)^\vee\!*_2\! f_2\right|
\end{equation*}
It then follows from the Cauchy-Schwarz inequality and the Young's inequality that
\begin{align*}
	\mathrm{II}&\lesssim \left\|\left|u\cdot\left(\psi_{k_1}\right)^\vee(u)\right|\!*_1\!|f_1|\right\|_2 \left\|\left(\psi_{k_2}\right)^\vee\!*_2\! f_2\right\|_2 \lesssim 2^{-k_1}\|f_1\|_2\|f_2\|_2\\
	&\lesssim 2^{-|k|/2}\|f_1\|_2\|f_2\|_2,
\end{align*}
where in the last inequality we have used $k\in \mathfrak{F}_H$.

By a similar argument we readily get that
\begin{equation*}
	\mathrm{III}\lesssim 2^{-|k|/2}\|f_1\|_2\|f_2\|_2.
\end{equation*}
This concludes the proof of \eqref{s2-6}.


\subsection{Estimation of  $T^L$ and $\sup_{|l|>\Gamma}|M_l^{L}|$.}\label{sec4-2}
In this section we will apply Christ, Durcik and Roos' \cite[Theorem 2]{CDR21}, which establishes the boundedness of a twisted bilinear operator.

Let $r_1,r_2$ be positive integers. Let $m$ be a smooth function on $\mathbb{R}^2\setminus\{(0,0)\}$ satisfying
\begin{equation}
	\left|\partial_{\xi_1}^{\alpha_1}\partial_{\xi_2}^{\alpha_2}m(\xi_1,\xi_2)\right|\lesssim_{\alpha_1,\alpha_2}\left(|\xi_1|^{1/r_1}+|\xi_2|^{1/r_2}\right)^{-(r_1\alpha_1+r_2\alpha_2)}\label{s2-9}
\end{equation}
for all $\alpha_1,\alpha_2\geq0$ up to a large finite order. For test functions $f_1,f_2$ on $\mathbb{R}^2$, let
\begin{equation*}
	T_m(f_1,f_2)(x,y)=\int_{\mathbb{R}^2}\!f_1(x+u,y)f_2(x,y+v)K(u,v)\,\mathrm{d}u\mathrm{d}v
\end{equation*}
with $K$ the distribution satisfying $m=\widehat{K}$.
\begin{lemma}[{Christ,  Durcik and Roos \cite[Theorem 2]{CDR21}}]\label{l2-1}
	Let $p,q\in(1,\infty)$, $r\in(1/2,2)$ be such that $1/p+1/q=1/r$. Assume that $m$ satisfies \eqref{s2-9}. Then $T_m$ extends to a bounded operator $L^p\times L^q\rightarrow L^r$.
\end{lemma}

Note that
\begin{equation*}
	M_l^L(f_1,f_2)=M_l\left(\chi_{\mathfrak{r}_1l+\mathcal{K}}^{\vee}\!*_1\!f_1,\chi_{\mathfrak{r}_2l+\mathcal{K}}^{\vee}\!*_2\!f_2\right).
\end{equation*}
and
\begin{equation*}
	T^L(f_1,f_2)=\sum_{|l|>\Gamma}T_l\left(\chi_{\mathfrak{r}_1l+\mathcal{K}}^{\vee}\!*_1\!f_1,\chi_{\mathfrak{r}_2l+\mathcal{K}}^{\vee}\!*_2\!f_2\right).
\end{equation*}
One can write them as
\begin{equation*}
	M_l^L(f_1,f_2)(x,y)=\int_{\mathbb{R}^2}\!f_1(u,y)f_2(x,v)K_1(x-u,y-v)\,\,\mathrm{d}u\mathrm{d}v,
\end{equation*}
\begin{equation*}
	T^L(f_1,f_2)(x,y)=\int_{\mathbb{R}^2}\!f_1(u,y)f_2(x,v)K_2(x-u,y-v)\,\,\mathrm{d}u\mathrm{d}v,
\end{equation*}
where
\begin{equation*}
	K_1(x,y)=2^{(\mathfrak{r}_1+\mathfrak{r}_2)l+2\mathcal{K}}\kappa_1\left(2^{\mathfrak{r}_1l+\mathcal{K}}x,2^{\mathfrak{r}_2l+\mathcal{K}}y\right),
\end{equation*}
\begin{equation*}
	K_2(x,y)=\sum_{|l|>\Gamma}2^{(\mathfrak{r}_1+\mathfrak{r}_2)l+2\mathcal{K}}\kappa_2\left(2^{\mathfrak{r}_1l+\mathcal{K}}x,2^{\mathfrak{r}_2l+\mathcal{K}}y\right)
\end{equation*}
and
\begin{equation*}
	\widehat{\kappa_i}(\xi_1,\xi_2)=\vartheta_i(\xi_1,\xi_2)\chi(\xi_1)\chi(\xi_2), \textrm{ for $i=1,2$,}
\end{equation*}
with
\begin{equation*}
	\vartheta_1(\xi_1,\xi_2)=\int_\mathbb{R}\!e^{2\pi i2^\mathcal{K}(\widetilde{P}_1(t)\xi_1+\widetilde{P}_2(t)\xi_2)}\psi(t)\,\mathrm{d}t,
\end{equation*}
\begin{equation*}
	\vartheta_2(\xi_1,\xi_2)=\int_\mathbb{R}\!e^{2\pi i2^\mathcal{K}(\widetilde{P}_1(t)\xi_1+\widetilde{P}_2(t)\xi_2)}\psi(t)t^{-1}\,\mathrm{d}t.
\end{equation*}

For $M_l^L$, the rapid decay of the Schwartz function $\kappa_1$ gives that
\begin{equation*}
	\left|M_l^L(f_1,f_2)(x,y)\right|\lesssim M_xf_1(x,y)M_yf_2(x,y),
\end{equation*}
where $M_x$ or $M_y$ is the Hardy-Littlewood maximal function applied in the $x$- or $y$-direction respectively. Thus by H\"{o}lder's inequality,
\begin{equation*}
	\left\|\sup_{|l|>\Gamma}|M_l^L(f_1,f_2)|\right\|_r\lesssim \|M_xf_1\|_p\|M_yf_2\|_q \lesssim \|f_1\|_p\|f_2\|_q
\end{equation*}
for all $p,q\in(1,\infty]$, $r\in(0,\infty]$ with $p^{-1}+q^{-1}=r^{-1}$.

For $T_L$, since the integral of $\psi(t)t^{-1}$ equals $0$, we have $\widehat{\kappa_2}(0,0)=0$ and for integers $\alpha_1,\alpha_2\geq0$,
\begin{equation*}
	\left|\partial_{\xi_1}^{\alpha_1}\partial_{\xi_2}^{\alpha_2}\widehat{\kappa_2}(\xi_1,\xi_2)\right|\leq C_{\alpha_1,\alpha_2,N,\mathcal{K}}(1+|\xi_1|+|\xi_2|)^{-N}
\end{equation*}
with $C_{\alpha_1,\alpha_2,N,\mathcal{K}}>0$ uniform in $l$. Hence one can verify that
\begin{equation*}
\left|\partial_{\xi_1}^{\alpha_1}\partial_{\xi_2}^{\alpha_2}\widehat{K_2}(\xi_1,\xi_2)\right|\lesssim_{\alpha_1,\alpha_2,\mathcal{K} }\left(|\xi_1|^{1/\mathfrak{r}_1}+|\xi_2|^{1/\mathfrak{r}_2}\right)^{-(\mathfrak{r}_1\alpha_1+\mathfrak{r}_2\alpha_2)}
\end{equation*}
for all $\alpha_1,\alpha_2\geq0$. By Lemma \ref{l2-1}, the operator $T^L$ extends to a bounded operator $L^p\times L^q\rightarrow L^r$ for all $p,q\in(1,\infty)$, $r\in[1,2)$ with $p^{-1}+q^{-1}=r^{-1}$.

\subsection{Estimation of  $T^M$ and $\sup_{|l|>\Gamma}|M_l^{M}|$.}\label{sec4-3}
The same argument as in Section \ref{sec4-2} shows that
\begin{equation*}
	\left\|\sup_{|l|>\Gamma}|M_l^M(f_1,f_2)|\right\|_r\lesssim \|M_xf_1\|_p\|M_yf_2\|_q \lesssim \|f_1\|_p\|f_2\|_q
\end{equation*}
for all $p,q\in(1,\infty]$, $r\in(0,\infty]$ with $p^{-1}+q^{-1}=r^{-1}$
and $T^M$ extends to a bounded operator $L^p\times L^q\rightarrow L^r$ for all $p,q\in(1,\infty)$, $r\in[1,2)$. This completes the proof of Theorem \ref{thm1}. \qed


\section{Proof of Theorem \ref{thm2}} \label{sec5}

In this section we combine Theorem \ref{thm3} and Bourgain's reduction in \cite{Bourgain88} to prove Theorem \ref{thm2}.  It suffices to prove that there is a constant $\delta$ with
\begin{equation}\label{s6-1}
\delta=\exp\left(-\exp\left(c\varepsilon^{-6}\right)\right)
\end{equation}
for some constant $c=c(P_1, P_2)>0$ such that
\begin{equation}
	I:=\int_{[0,1]^3} \!\!  f(x,y)f\left(x+P_1(t),y\right)f\left(x,y+P_2(t)\right) \,\mathrm{d}t\mathrm{d}x\mathrm{d}y>\delta\label{s6-2}
\end{equation}
for all measurable functions $f$ on $\mathbb{R}^2$ with $\supp(f)\subset [0, 1]^2$, $0\leq f\leq 1$ and $\int_{[0,1]^2} \! f\geq \varepsilon$. Then the desired result follows easily by taking $f=\textbf{1}_{S}$.

Let $\tau$ be a nonnegative smooth bump function supported on $[1/2,2]$ with $\int\!\tau=1$. Let $\rho$ be a nonnegative radially decreasing smooth bump function which is constant on $[-1,1]$ with $\supp(\rho)\subset [-2,2]$ and $\int\!\rho=1$. For $l\in\mathbb{Z}$, set $\tau_l(t)=2^l\tau(2^l t)$ and $\rho_l(t)=2^l\rho(2^l t)$.

For $\Gamma, l, l', l''\in\mathbb{N}$ with $\Gamma \leq l'\ll l\ll l''$,
we have a decomposition
\begin{align*}
	2^l I&\gtrsim_\tau\int_{[0,1]^3} \!\!  f(x,y)f\left(x+P_1(t),y\right)f\left(x,y+P_2(t)\right) \tau_l(t)\,\mathrm{d}t\mathrm{d}x\mathrm{d}y\\
	&=I_1+I_2+I_3,
\end{align*}
where
\begin{align*}
	I_1&=\!\!\int_{[0,1]^3} \!\!\!  f(x,y)f\left(x+P_1(t),y\right)\rho_{l'}\!*_2\!f\left(x,y+P_2(t)\right)\tau_l(t) \mathrm{d}t\mathrm{d}x\mathrm{d}y,\\
	I_2&=\!\!\int_{[0,1]^3} \!\!\!  f(x,y)f\left(x+P_1(t),y\right)(\rho_{l''}\!*_2\!f-\rho_{l'}\!*_2\!f)\left(x,y+P_2(t)\right)\tau_l(t)\mathrm{d}t\mathrm{d}x\mathrm{d}y,\\
	I_3&=\!\!\int_{[0,1]^3} \!\!\!  f(x,y)f\left(x+P_1(t),y\right)(f-\rho_{l''}\!*_2\!f)\left(x,y+P_2(t)\right)\tau_l(t) \mathrm{d}t\mathrm{d}x\mathrm{d}y.
\end{align*}
By the Cauchy-Schwarz inequality, it is easy to get
\begin{align*}
	|I_2|\leq\|\rho_{l''}\!*_2\!f-\rho_{l'}\!*_2\!f\|_2.
\end{align*}

To estimate $I_3$, we consider a dyadic decomposition
\begin{equation*}
	f-\rho_{l''}\!*_2 f=S_{\lfloor k_0\rfloor}^{(2)}(f-\rho_{l''}\!*_2\!f)+\sum_{k>k_0}\Delta_k^{(2)}(f-\rho_{l''}\!*_2\!f)
\end{equation*}
with a parameter $k_0>0$ to be chosen below. Then we write $I_3$ as
\begin{equation*}
	I_3=I_4+\sum_{k>k_0}I_{3,k},
\end{equation*}
where
\begin{equation*}
	I_4=\!\!\int_{[0,1]^3} \!\!  f(x,y)f\left(x+P_1(t),y\right)S_{\lfloor k_0\rfloor}^{(2)}(f-\rho_{l''}\!*_2\!f)\left(x,y+P_2(t)\right)\tau_l(t) \mathrm{d}t\mathrm{d}x\mathrm{d}y
\end{equation*}
and
\begin{equation*}
	I_{3,k}=\!\!\int_{[0,1]^3} \!\!  f(x,y)f\left(x+P_1(t),y\right)\Delta_{k}^{(2)}(f-\rho_{l''}\!*_2\!f)\left(x,y+P_2(t)\right)\tau_l(t) \mathrm{d}t\mathrm{d}x\mathrm{d}y.
\end{equation*}
By the Cauchy-Schwarz inequality and the Plancherel theorem, we have
\begin{equation*}
	|I_4|\leq\|f\|_2\left\|S_{\lfloor k_0\rfloor}^{(2)}(f-\rho_{l''}\!*_2\!f)\right\|_2\lesssim 2^{k_0-l''}.
\end{equation*}
For each $k>k_0$, with $g_k=\Delta_{k}^{(2)}(f-\rho_{l''}\!*_2\!f)$, by rescaling and adding a partition of unity we have
\begin{align*}
	|I_{3,k}|\leq &2^{-(\sigma_1+\sigma_2)l}\cdot\\
	&\sum_{R\in\mathcal{R}_l}\int\!\left|\int\!\widetilde{f}(x+\widetilde{P}_1(t),y)\widetilde{g}_k(x,y+\widetilde{P}_2(t))\zeta_R(x,y)\tau(t)\,\mathrm{d}t\right|\mathrm{d}x\mathrm{d}y,
\end{align*}
where
\begin{equation*}
	\widetilde{f}(x,y)=f\left(2^{-\sigma_1l}x,2^{-\sigma_2l}y\right),\quad  \widetilde{g}_k(x,y)=g_k\left(2^{-\sigma_1l}x,2^{-\sigma_2l}y\right),
\end{equation*}
$\mathcal{R}_l$ is the family of almost disjoint unit squares that form a partition of the set $[0,2^{\sigma_1 l}]\times[0,2^{\sigma_2 l}]$ and, for each $R\in\mathcal{R}_l$, $\zeta_R$ is a nonnegative smooth bump function supported in a neighborhood of $R$ such that $\sum_{R\in\mathcal{R}_l}\zeta_R(x,y)=1$ on $[0,2^{\sigma_1 l}]\times[0,2^{\sigma_2 l}]$.

If $l>\Gamma$ is sufficiently large, applying Theorem \ref{thm3} (with $\lambda=2^{k-\sigma_2 l}$) gives that
\begin{align*}
	|I_{3,k}|&\lesssim2^{-(\sigma_1+\sigma_2)l}\sum_{R\in\mathcal{R}_l}2^{\mathfrak{b}l}2^{-\sigma(k-\sigma_2 l)}\left\|\widetilde{f}\right\|_2\left\|\widetilde{g}_k\right\|_2\\
	&\lesssim 2^{(\mathfrak{b}+\sigma_1+(\sigma+1)\sigma_2)l-\sigma k}.
\end{align*}
To sum up, by choosing a proper $k_0$, we thus get
\begin{equation*}
	|I_3|\lesssim 2^{k_0-l''}+2^{(\mathfrak{b}+\sigma_1+(\sigma+1)\sigma_2)l-\sigma k_0}\lesssim 2^{\mathfrak{b}_1l-\mathfrak{b}_2l''}
\end{equation*}
for some fixed constants $\mathfrak{b}_1,\mathfrak{b}_2>0$.

To estimate $I_1$, we set
\begin{equation*}
	I_1':=\int_{[0,1]^2} \!\!  f(x,y)\rho_{l'}\!*_2\!f\left(x,y\right)\left(\int _\mathbb{R}\! f\left(x+P_1(t),y\right)\tau_l(t) \,\mathrm{d}t\right)\mathrm{d}x\mathrm{d}y.
\end{equation*}
Then by the mean value theorem, we have
\begin{equation*}
	I_1-I_1'=O_{P_2}\left(2^{l'-l}\right).
\end{equation*}
Notice that the inner integral
\begin{equation}
	\int _\mathbb{R}\! f\left(x+P_1(t),y\right)\tau_l(t) \,\mathrm{d}t \label{s5-23}
\end{equation}
is in fact over $t\asymp2^{-l}$. If $l>\Gamma$ is sufficiently large, the size of $P_1(t)$ is dominated by its monomial $a_{\sigma_1}t^{\sigma_1}$. We use the substitution
\begin{equation*}
	\omega=|P_1(t)|
\end{equation*}
to rewrite the integral \eqref{s5-23} as a convolution. We may assume that $a_{\sigma_1}<0$ while the case $a_{\sigma_1}>0$ is the same up to a reflection. Hence
\begin{equation*}
	\eqref{s5-23}=\widetilde{\tau}\!*_1\!f(x,y),
\end{equation*}
where
\begin{equation*}
	\widetilde{\tau}(\omega)=\tau_l(t(\omega))t'(\omega).
\end{equation*}
Let $\varsigma_l=|a_{\sigma_1}|2^{-\sigma_1l}$ and $\rho_{\varsigma_l}(x)=\varsigma_l^{-1}\rho(\varsigma_l^{-1}x)$. Then
\begin{align*}
	\left\|\widetilde{\tau}\!*_1\!f-\rho_{\varsigma_{l'}}\!*_1\!f\right\|_2
	&\leq \left\|\rho_{\varsigma_{l''}}\!*_1\!f-\rho_{\varsigma_{l'}}\!*_1\!f\right\|_2 \\
	&\quad +\left\|\widetilde{\tau}-\widetilde{\tau}*\rho_{\varsigma_{l''}}\right\|_1+ \left\|\widetilde{\tau}*\rho_{\varsigma_{l'}}-\rho_{\varsigma_{l'}}\right\|_1\\
	&=\left\|\rho_{\varsigma_{l''}}\!*_1\!f-\rho_{\varsigma_{l'}}\!*_1\!f\right\|_2+O\left(2^{l-l''}\right)+O\left(2^{l'-l}\right).
\end{align*}
The last two bounds follow from rescaling and the mean value theorem. We thus have
\begin{equation*}
	\left|I_1'-I_1''\right|\leq \left\|\rho_{\varsigma_{l''}}\!*_1\!f-\rho_{\varsigma_{l'}}\!*_1\!f\right\|_2+O\left(2^{l-l''}\right)+O\left(2^{l'-l}\right),
\end{equation*}
where
\begin{equation*}
	I_1'':=\int_{[0,1]^2} \!\!  f(x,y)\rho_{l'}\!*_2\!f\left(x,y\right)\rho_{\varsigma_{l'}}\!*_1\!f(x,y)\,\mathrm{d}x\mathrm{d}y.
\end{equation*}
By  \cite[Lemma 5.1]{CDR21}, an analogue of Bourgain's \cite[Lemma 6]{Bourgain88},
\begin{equation*}
	I_1''\geq c_{\rho} \left(\int_{[0,1]^2} \! f \right)^3\geq c_{\rho}\varepsilon^3.
\end{equation*}


Collecting the above upper and lower bounds yields that if $l''$ (resp. $l$) is large enough with respect to $l$ (resp. $l'$) then
\begin{equation*}
	2^l I+\left\|\rho_{l''}\!*_2\!f-\rho_{l'}\!*_2\!f\right\|_2+\left\|\rho_{\varsigma_{l''}}\!*_1\!f-\rho_{\varsigma_{l'}}\!*_1\!f\right\|_2\geq c\varepsilon^3.
\end{equation*}
In fact we can choose a sequence $\Gamma=l_1<l_2<\cdots<l_k<\cdots$ (independently of $f$) such that for each $k\in\mathbb{N}$ we have $l_{k+1}\asymp(\mathfrak{b}_1/\mathfrak{b}_2)^k\log \varepsilon^{-1}$ and that either
\begin{equation}
	I>2^{-l_{k+1}-1}c\varepsilon^3   \label{s6-3}
\end{equation}
or
\begin{equation}
	\left\|\rho_{l_{k+1}}\!*_2\!f-\rho_{l_{k}}\!*_2\!f\right\|_2+\left\|\rho_{\varsigma_{l_{k+1}}}\!*_1\!f-\rho_{\varsigma_{l_k}}\!*_1\!f\right\|_2\geq c\varepsilon^{3}/2. \label{s6-4}
\end{equation}
Note that by using the Plancherel theorem and the fast decay of $\widehat{\rho}$ we have
\begin{equation*}
	\sum_{k=1}^{\infty}\left(\left\|\rho_{l_{k+1}}\!*_2\!f-\rho_{l_{k}}\!*_2\!f\right\|_2^2+\left\|\rho_{\varsigma_{l_{k+1}}}\!*_1\!f-\rho_{\varsigma_{l_k }}\!*_1\!f\right\|_2^2\right)\leq C_{\rho}.
\end{equation*}
Thus \eqref{s6-4} can only occur a bounded number of times and \eqref{s6-3} must hold for some $1\leq k_0\leq K:=\lceil 12c^{-2}C_\rho\varepsilon^{-6}\rceil+1$. Therefore
\begin{equation*}
	I>2^{-l_{k_0+1}-1}c\varepsilon^3\geq2^{-l_{K+1}-1}c\varepsilon^3.
\end{equation*}
Using the size estimate of $l_{K+1}$, we conclude that there exists a constant $\delta=\delta(\varepsilon,P_1,P_2)$ satisfying \eqref{s6-1} and \eqref{s6-2}. \qed


\appendix

\section{Polynomial factorization}

\begin{lemma}\label{app-1}
If nonnegative integers $\alpha$, $\beta$ and $\gamma$ satisfy $\alpha>\beta>\gamma$, then
\begin{align*}
 &x^{\alpha}y^{\beta}z^{\gamma}+x^{\beta}y^{\gamma}z^{\alpha}+x^{\gamma}y^{\alpha}z^{\beta}
-x^{\beta}y^{\alpha}z^{\gamma}-x^{\gamma}y^{\beta}z^{\alpha}-x^{\alpha}y^{\gamma}z^{\beta}\\
=&(x-y)(x-z)(y-z)\sum_{i=0}^{\beta-\gamma-1}\sum_{j=0}^{\alpha-\beta-1}\sum_{k=0}^{\alpha-\beta-1-j+i}
x^{\gamma+i+j}y^{\beta-1-i+k}z^{\alpha-2-j-k}.
\end{align*}

If two of nonnegative indices $\alpha$, $\beta$ and $\gamma$ are equal, then
\begin{equation*}
x^{\alpha}y^{\beta}z^{\gamma}+x^{\beta}y^{\gamma}z^{\alpha}+x^{\gamma}y^{\alpha}z^{\beta}
-x^{\beta}y^{\alpha}z^{\gamma}-x^{\gamma}y^{\beta}z^{\alpha}-x^{\alpha}y^{\gamma}z^{\beta}=0.
\end{equation*}
\end{lemma}

\begin{remark}
The triple sum on the right side of the first identity is a homogeneous polynomial of degree $\alpha+\beta+\gamma-3$ with all coefficients being positive.
\end{remark}

\begin{proof}[Proof of Lemma \ref{app-1}]
The second identity is obvious. We will prove the first one by repeatedly using the identity
\begin{equation*}
x^m-y^m=(x-y)\sum_{i=0}^{m-1} x^i y^{m-1-i}.
\end{equation*}

We first get
\begin{equation*}
x^{\alpha}y^{\beta}z^{\gamma}-x^{\beta}y^{\alpha}z^{\gamma}=
(x-y)x^{\beta}y^{\beta}z^{\gamma}\sum_{i=0}^{\alpha-\beta-1}x^i y^{\alpha-\beta-1-i},
\end{equation*}
\begin{equation*}
x^{\beta}y^{\gamma}z^{\alpha}-x^{\gamma}y^{\beta}z^{\alpha}=
(x-y)x^{\gamma}y^{\gamma}z^{\alpha}\sum_{i=0}^{\beta-\gamma-1}x^i y^{\beta-\gamma-1-i}
\end{equation*}
and
\begin{align*}
 &x^{\gamma}y^{\alpha}z^{\beta}-x^{\alpha}y^{\gamma}z^{\beta}\\
=&-(x-y)x^{\gamma}y^{\gamma}z^{\beta}\sum_{i=0}^{\alpha-\gamma-1}x^i y^{\alpha-\gamma-1-i}\\
=&-(x-y)\bigg( x^{\gamma}z^{\beta}\sum_{i=0}^{\alpha-\beta-1}x^i y^{\alpha-1-i}+x^{\alpha-\beta+\gamma}z^{\beta}\sum_{i=0}^{\beta-\gamma-1}x^i y^{\beta-1-i}   \bigg).
\end{align*}
Adding the three identities above gives
\begin{align*}
 &x^{\alpha}y^{\beta}z^{\gamma}+x^{\beta}y^{\gamma}z^{\alpha}+x^{\gamma}y^{\alpha}z^{\beta}
-x^{\beta}y^{\alpha}z^{\gamma}-x^{\gamma}y^{\beta}z^{\alpha}-x^{\alpha}y^{\gamma}z^{\beta}\\
=&(x-y)\bigg(\!\! (x^{\beta}z^{\gamma}-x^{\gamma}z^{\beta})\!\!\sum_{i=0}^{\alpha-\beta-1}\!\!x^i y^{\alpha-1-i}\!-\!(x^{\alpha-\beta+\gamma}z^{\beta}-x^{\gamma}z^{\alpha})\!\!\sum_{i=0}^{\beta-\gamma-1}\!\!x^i y^{\beta-1-i}\!   \bigg)\\
=&(x-y)(x-z)\cdot\\
 &\bigg( \sum_{i=0}^{\alpha-\beta-1}\sum_{j=0}^{\beta-\gamma-1}x^{\gamma+i+j}y^{\alpha-1-i}z^{\beta-1-j}-
\sum_{j=0}^{\alpha-\beta-1}\sum_{i=0}^{\beta-\gamma-1}x^{\gamma+i+j}y^{\beta-1-i}z^{\alpha-1-j}\bigg)\\
=&(x-y)(x-z)\sum_{j=0}^{\alpha-\beta-1}\sum_{i=0}^{\beta-\gamma-1}x^{\gamma+i+j}
\left(y^{\alpha-1-j}z^{\beta-1-i}-y^{\beta-1-i}z^{\alpha-1-j}\right)\\
=&(x-y)(x-z)(y-z)\sum_{i=0}^{\beta-\gamma-1}\sum_{j=0}^{\alpha-\beta-1}\sum_{k=0}^{\alpha-\beta-1-j+i}
x^{\gamma+i+j}y^{\beta-1-i+k}z^{\alpha-2-j-k},
\end{align*}
as desired.
\end{proof}

\section{The inverse function theorem}

The following is a quantitative version of the inverse function theorem.

\begin{lemma} \label{app-2}
Suppose that $f$ is a $C^{(k)}$ ($k\geqslant 2$) mapping from an
open set $\Omega\subset\mathbb{R}^d$ into $\mathbb{R}^d$. For some $a\in \Omega$, assume that
\begin{equation*}
|\det (\nabla f(a))|\geqslant c>0
\end{equation*}
and
\begin{equation*}
|D^{\nu} f_i(x)|\leqslant C \quad \quad \textrm{for all $x\in\Omega$, $|\nu|\leqslant
2$, $1\leqslant i\leqslant d$}.
\end{equation*}
If $\mathfrak{r}_0\leqslant \sup\{r>0: B(a, r)\subset \Omega\}$ then $f$ is
bijective from $B(a, r_1)$ to an open set $f(B(a, r_1))$ such that
\begin{equation*}
B(f(a), r_2)\subset f(B(a, r_1))\subset B(f(a), r_3),
\end{equation*}
where
\begin{equation*}
r_1=\min\left\{\frac{c}{2d^{2} d! C^d}, r_0\right\},
r_2=\frac{c}{4d!C^{d-1}}r_1 \textrm{ and }  r_3=\sqrt{d}C r_1.
\end{equation*}
The inverse mapping $f^{-1}$ is also in $C^{(k)}$.
\end{lemma}

\end{document}